\documentclass[11pt,reqno,oneside]{amsproc}
\title[Fluid-structure interaction problem]{On the local existence of solutions to the fluid-structure interaction problem\\with a free interface}

\author[I.~Kukavica]{\colu Igor Kukavica}
\address{Department of Mathematics, University of Southern California, Los Angeles, CA 90089}
\email{kukavica@usc.edu}

\author[L.~Li]{\colu Linfeng Li}
\address{Department of Mathematics, University of California Los Angeles, Los Angeles, CA 90095}
\email{lli265@math.ucla.edu}

\author[A.~Tuffaha]{\colu Amjad Tuffaha}
\address{Department of Mathematics and Statistics, American University of Sharjah, Sharjah, UAE}
\email{atufaha\char'100aus.edu}

\chardef\forshowkeys=0
\chardef\showllabel=0
\chardef\refcheck=0
\chardef\sketches=0
\chardef\showcolors=0
  
\usepackage{enumitem}
\usepackage{fancyhdr}
\usepackage{comment}

\ifnum\forshowkeys=1

\usepackage[notref,notcite,color]{showkeys}
\fi

\usepackage[margin=1in]{geometry}
\usepackage{amsmath, amsthm, amssymb}
\usepackage{times}
\usepackage{graphicx}
\usepackage[usenames,dvipsnames,svgnames,table]{xcolor}
\usepackage{marginnote}
\usepackage[unicode,breaklinks=true,colorlinks=true,linkcolor=blue,urlcolor=blue,citecolor=blue]{hyperref}

\usepackage{tikz}

\ifnum\refcheck=1
\usepackage{refcheck}
\fi

\begin{document}

\def\CCC{\bar C}
\def\and{\quad\text{and}\quad}
\def\tLambda{\tilde \Lambda}
\def\tPi{\tilde \Pi}
\def\hff{h_{1}}
\def\hf{h_{2}}
\def\XX{X}
\def\YY{Y}
\def\ZZZ{Z}

\def\intint{\int\!\!\!\!\int}
\def\OO{\mathcal O}
\def\SS{\mathbb S}
\def\CC{\mathbb C}
\def\RR{\mathbb R}
\def\TT{\tilde{T}}
\def\ZZ{\mathbb Z}
\def\RSZ{\mathcal R}
\def\LL{\mathcal L}
\def\SL{\LL^1}
\def\ZL{\LL^\infty}
\def\GG{\mathcal G}

\def\tt{\langle t\rangle}
\def\erf{\mathrm{Erf}}
\def\red#1{\textcolor{red}{#1}}
\def\blue#1{\textcolor{blue}{#1}}
\def\mgt#1{\textcolor{magenta}{#1}}
\def\ff{\rho}
\def\gg{G}
\def\sqrtnu{\sqrt{\nu}}
\def\ww{w}
\def\ft#1{#1_\xi}
\def\lec{\lesssim}
\def\ges{\gtrsim}
\renewcommand*{\Re}{\ensuremath{\mathrm{{\mathbb R}e\,}}}
\renewcommand*{\Im}{\ensuremath{\mathrm{{\mathbb I}m\,}}}

\ifnum\showllabel=1
\def\llabel#1{\marginnote{\color{lightgray}\rm\small(#1)}[-0.0cm]\notag}
\else
\def\llabel#1{\notag}
\fi

\newcommand{\norm}[1]{\left\|#1\right\|}
\newcommand{\nnorm}[1]{\lVert #1\rVert}
\newcommand{\abs}[1]{\left|#1\right|}
\newcommand{\NORM}[1]{|\!|\!| #1|\!|\!|}

\newtheorem{Theorem}{Theorem}[section]
\newtheorem{Corollary}[Theorem]{Corollary}
\newtheorem{Definition}[Theorem]{Definition}
\newtheorem{Proposition}[Theorem]{Proposition}
\newtheorem{Lemma}[Theorem]{Lemma}
\newtheorem{Remark}[Theorem]{Remark}

\def\theequation{\thesection.\arabic{equation}}
\numberwithin{equation}{section}

\definecolor{mygray}{rgb}{.6, .6, .6}
\definecolor{myblue}{rgb}{9, 0, 1}
\definecolor{colorforkeys}{rgb}{1.0,0.0,0.0}

\newlength\mytemplen
\newsavebox\mytempbox

\makeatletter
\newcommand\mybluebox{%
\@ifnextchar[
{\@mybluebox}%
{\@mybluebox[0pt]}}

\def\@mybluebox[#1]{%
\@ifnextchar[
{\@@mybluebox[#1]}%
{\@@mybluebox[#1][0pt]}}

\def\@@mybluebox[#1][#2]#3{
\sbox\mytempbox{#3}%
\mytemplen\ht\mytempbox
\advance\mytemplen #1\relax
\ht\mytempbox\mytemplen
\mytemplen\dp\mytempbox
\advance\mytemplen #2\relax
\dp\mytempbox\mytemplen
\colorbox{myblue}{\hspace{1em}\usebox{\mytempbox}\hspace{1em}}}

\makeatother

\ifnum\showcolors=1
  \def\colr{\color{red}}
  \def\colrr{\color{black}}
  \def\colb{\color{black}}
  \def\coly{\color{lightgray}}
  \definecolor{colorgggg}{rgb}{0.1,0.5,0.3}
  \definecolor{colorllll}{rgb}{0.0,0.7,0.0}
  \definecolor{colorhhhh}{rgb}{0.3,0.75,0.4}
  \definecolor{colorpppp}{rgb}{0.7,0.0,0.2}
  \definecolor{coloroooo}{rgb}{0.45,0.0,0.0}
  \definecolor{colorqqqq}{rgb}{0.1,0.7,0}
  \def\colg{\color{colorgggg}}
  \def\collg{\color{colorllll}}
  \def\cole{\color{coloroooo}}
  \def\coleo{\color{colorpppp}}
  \def\colu{\color{blue}}
  \def\colc{\color{colorhhhh}}
  \def\colW{\colb}   
  \definecolor{coloraaaa}{rgb}{0.6,0.6,0.6}
  \def\colw{\color{coloraaaa}}
\else
  \def\colr{\color{black}}
  \def\colrr{\color{black}}
  \def\colb{\color{black}}
  \def\coly{\color{black}}
  \def\colg{\color{black}}
  \def\collg{\color{black}}
  \def\cole{\color{black}}
  \def\coleo{\color{black}}
  \def\colu{\color{blue}}
  \def\colc{\color{black}}
  \def\colW{\color{black}}
  \def\colw{\color{black}}
\fi

\def\Gac{\Gamma_{\text{c}}}
\def\bnew{\colr}
\def\enew{\colb}
\def\bold{\colu}
\def\eold{\colb}
\def\uu{\tilde u}
\def\aa{\theta}
\def\KK{K}
\def\ll{\lambda}
\def\dwdn{\frac{\partial w}{\partial N}}
\def\dwbardn{\frac{\partial \bar w}{\partial N}}
\def\dwbardntext{\fractext{\partial \bar w}{\partial N}}
\def\dwdntext{\fractext{\partial w}{\partial N}}
\def\un{u^{(n)}}
\def\AA{\tilde g}
\def\BB{b}

\def\HH{H_{\text f}}
\def\Omegaf{\Omega_{\text f}}
\def\Omegae{\Omega_{\text e}}
\def\Gammac{\Gamma_{\text c}}
\def\Gammaf{\Gamma_{\text f}}
\def\thf{{\tilde h}_{\text f}}
\def\cof{\mathop{\rm cof}\nolimits}
\def\OPS{\mathop{\rm OPS}\nolimits}
\def\Dn{\frac{\partial}{\partial N}}
\def\Dnn#1{\frac{\partial #1}{\partial N}}
\def\lat{\Delta_2}
\def\biglinem{\vskip0.5truecm\par==========================\par\vskip0.5truecm}

\def\inon#1{\hbox{\ \ \ \ \ \ \ }\hbox{#1}}                
\def\onon#1{\inon{on~$#1$}}
\def\inin#1{\inon{in~$#1$}}

\def\FF{F}
\def\ww{w(y)}
\def\ee{\epsilon_0}
\def\startnewsection#1#2{\medskip\section{#1}\label{#2}\setcounter{equation}{0}\medskip}   
\def\nnewpage{ }
\def\sgn{\mathop{\rm sgn\,}\nolimits}    
\def\Tr{\mathop{\rm Tr}\nolimits}    
\def\div{\mathop{\rm div}\nolimits}
\def\curl{\mathop{\rm curl}\nolimits}
\def\dist{\mathop{\rm dist}\nolimits}  
\def\supp{\mathop{\rm supp}\nolimits}
\def\indeq{\quad{}}           
\def\period{.}                       
\def\scl{,}
\def\nts#1{{\colr #1\colb}}
\def\ntsik#1{{\colr IK:~#1\colb}}
\def\ques#1{{\colr #1\colb}}
\def\pt{\partial_t}

\def\comma{ {\rm ,\qquad{}} }            
\def\commaone{ {\rm ,\quad{}} }          
\def\les{\lesssim}
\def\nts#1{{\color{red}\hbox{\bf ~#1~}}} 
\def\ntsf#1{\footnote{\color{colorgggg}\hbox{#1}}} 
\def\blackdot{{\color{red}{\hskip-.0truecm\rule[-1mm]{4mm}{4mm}\hskip.2truecm}}\hskip-.3truecm}
\def\bluedot{{\color{blue}{\hskip-.0truecm\rule[-1mm]{4mm}{4mm}\hskip.2truecm}}\hskip-.3truecm}
\def\purpledot{{\color{colorpppp}{\hskip-.0truecm\rule[-1mm]{4mm}{4mm}\hskip.2truecm}}\hskip-.3truecm}
\def\greendot{{\color{colorgggg}{\hskip-.0truecm\rule[-1mm]{4mm}{4mm}\hskip.2truecm}}\hskip-.3truecm}
\def\cyandot{{\color{cyan}{\hskip-.0truecm\rule[-1mm]{4mm}{4mm}\hskip.2truecm}}\hskip-.3truecm}
\def\reddot{{\color{red}{\hskip-.0truecm\rule[-1mm]{4mm}{4mm}\hskip.2truecm}}\hskip-.3truecm}

\def\gdot{\greendot}
\def\bdot{\bluedot}
\def\tdot{\greendot}
\def\ydot{\cyandot}
\def\rdot{\cyandot}

\def\fractext#1#2{{#1}/{#2}}
\def\II{\mathcal{I}}
\def\ii{\hat\imath}
\def\fei#1{\textcolor{blue}{#1}}
\def\vlad#1{\textcolor{cyan}{#1}}
\def\igor#1{\text{{\textcolor{colorqqqq}{#1}}}}
\def\igorf#1{\footnote{\text{{\textcolor{colorqqqq}{#1}}}}}

\newcommand{\UE}{U^{\rm E}}
\newcommand{\PE}{P^{\rm E}}
\newcommand{\KP}{K_{\rm P}}
\newcommand{\uNS}{u^{\rm NS}}
\newcommand{\vNS}{v^{\rm NS}}
\newcommand{\pNS}{p^{\rm NS}}
\newcommand{\omegaNS}{\omega^{\rm NS}}
\newcommand{\uE}{u^{\rm E}}
\newcommand{\vE}{v^{\rm E}}
\newcommand{\pE}{p^{\rm E}}
\newcommand{\omegaE}{\omega^{\rm E}}
\newcommand{\ua}{u_{\rm   a}}
\newcommand{\va}{v_{\rm   a}}
\newcommand{\omegaa}{\omega_{\rm   a}}
\newcommand{\ue}{u_{\rm   e}}
\newcommand{\ve}{v_{\rm   e}}
\newcommand{\omegae}{\omega_{\rm e}}
\newcommand{\omegaeic}{\omega_{{\rm e}0}}
\newcommand{\ueic}{u_{{\rm   e}0}}
\newcommand{\veic}{v_{{\rm   e}0}}
\newcommand{\up}{u^{\rm P}}
\newcommand{\vp}{v^{\rm P}}
\newcommand{\tup}{{\tilde u}^{\rm P}}
\newcommand{\bvp}{{\bar v}^{\rm P}}
\newcommand{\omegap}{\omega^{\rm P}}
\newcommand{\tomegap}{\tilde \omega^{\rm P}}
\renewcommand{\up}{u^{\rm P}}
\renewcommand{\vp}{v^{\rm P}}
\renewcommand{\omegap}{\Omega^{\rm P}}
\renewcommand{\tomegap}{\omega^{\rm P}}

\begin{abstract}
We address a system of equations modeling an incompressible fluid interacting with an elastic body. 
We prove 
the local existence when the initial velocity belongs to the space
$H^{1.5+\epsilon}$ and the initial structure velocity is in $H^{1+\epsilon}$, where $\epsilon \in (0, 1/20)$.
\hfill \today
\end{abstract}

\maketitle

\setcounter{tocdepth}{1} 
\tableofcontents

\startnewsection{Introduction}{sec01}

The purpose of this paper is to establish the local-in-time existence of solutions for the free boundary fluid-structure interaction model with minimal regularity assumptions on the initial data.  The model, which first appeared in a modeling book \cite{MZ}, describes the interaction between an elastic structure and a viscous incompressible fluid in which the structure is immersed.  Mathematically, the dynamics of the fluid are captured by the incompressible Navier-Stokes equations in the velocity and pressure variables $(u,p)$, while the elastic dynamics are described by a second order elasticity equation (we replace it with a wave equation to simplify presentation; see Remark~\ref{R01}) in the vector variables $(w,w_{t})$ representing the displacement and velocity of the structure. The interaction between the structure and the fluid is mathematically described by velocity and stress matching boundary conditions at the moving interface separating the solid and fluid regions.  Since the interface position evolves with time and is not known a~priori, this is a free-boundary problem.

Well-posedness and local-in-time existence were first obtained by Coutand and Shkoller \cite{CS1, CS2} in 2005. The authors use the Lagrangian coordinate system to fix the domain and a Tychonoff fixed point theorem to construct local-in-time solutions given initial fluid velocity $u_{0} \in H^{5}$ and structural velocity $w_{1} \in H^{3}$.  A key feature of the model is the mismatch between parabolic and hyperbolic regularity and the importance of trace regularity theorems for both parabolic and hyperbolic equations in establishing the regularity of solutions. In \cite{KT1, KT2}, the authors established a~priori estimates for the local existence of solutions using direct estimates given $u_{0} \in H^{3}$ for the initial velocity and $w_{1} \in H^{5/2+r}$, where $r \in (0, (\sqrt{2}-1)/2)$, for the initial structural velocity. A key ingredient in obtaining the result was the hidden regularity trace theorems for the wave equations established in \cite{LLT, BL, L1, L2,S,T}. A wealth of literature on wave-heat coupled systems on a non-moving domain was instrumental in further understanding the heat-wave interaction phenomena (cf.~\cite{ALT, AT1, AT2, DGHL, BGLT1, BGLT2, KTZ1, KTZ2, KTZ3, LL1}). Similar results were obtained for the compressible free-boundary version~\cite{BG1, BG2, KT3,KLT1}.  For some other works on fluid-structure models, cf.~\cite{B,BuL,BZ1,BZ2,BTZ,DEGL,F,GH,GGCC,GGCCL,IKLT1,LL1,LL2,LT,LTr1,LTr2,MC1,MC2,MC3,SST}.

More recently, a sharp regularity result for the case when the initial domain is a flat channel was obtained by Raymond and Vanninathan~\cite{RV}. The authors study the system in the Lagrangian coordinate setting and obtain local-in-time solutions for the 3D model when $u_{0} \in H^{1.5+\epsilon}$ and $w_{1} \in H^{1+\epsilon+\delta}$, where $\epsilon\in (0,1/2)$ and $\delta>0$.
More recently, Boulakia, Guerrero, and Takahashi obtained in \cite{BGT} a local-in-time solution given data $u_{0} \in H^{2}$ and $w_{1} \in H^{9/8}$ for the case of a general domain.

In this paper, we provide a natural proof of the existence of unique local-in-time solutions to the system under the sharp assumption of $u_{0} \in H^{1.5+\epsilon}$ and $w_{1} \in H^{1+\epsilon}$, where $\epsilon\in (0,1/20)$, in the case of a flat domain.  
Note that this regularity seems to be the lowest regularity which guarantees the cofactor matrix to be bounded.  Our proof relies on a maximal regularity type theorem for the linear Stokes system with Neumann type conditions and nonhomogeneous divergence in addition to the hidden regularity theorems for the wave equation. 
Another essential ingredient in the proof of the main theorem is a trace inequality 
for functions which are Sobolev in the time variable and square integrable on the boundary (cf.~Lemma~\ref{L01} below).  
This is used essentially in the proof of the main results, Theorem~\ref{T01}.

When establishing the existence of solutions, we encounter a fundamental difficulty that the constants in inequalities we use inversely proportional to powers of time $T$ for $T$ small. This then poses  problems in the convergence of a fixed-point scheme for small time. 
The main difficulty is then that one cannot benefit from the small time in a low Sobolev regularity Poincar\'e-type inequalities (see Proposition~\ref{P01} below).
In this paper we overcome this difficulty by modifying the system using a cutoff time function of the integrated velocity matching conditions (see~\eqref{EQ320} below) and appealing to the hidden trace regularity Lemma~\ref{L03} for the normal derivative.
A major advantage of the new approach using the cut-offs is that the interpolation inequality, trace, hidden regularity, and maximal regularity theorems are all applied on a fixed space-time domain and thus the implicit constants appeared are independent of the small time.
We benefit from the small-time Poincar\'e type inequalities by making sure that they are always used
when the cut-off is present and when the Sobolev spaces have an integer index.
\colb

The primary challenge of establishing the fixed-point theorems (for both the linear and nonlinear variants) is that the time derivatives, which are often fractional, fall on the cutoff function, showing that the careful analysis of the cutoff function and the powers of~$\TT$ is required.
Moreover, in the nonlinear system treated in Section~\ref{sec06}, we also modify the Lagrangian flow map and its cofactor matrix using the cutoff function to guarantee the contraction-type estimates on the solution map for the system with given variable coefficients.

The construction of solutions for the fluid-structure problem is obtained via the Banach fixed point theorem. The scheme depends on solving the Stokes system with the variable coefficients treated as a given forcing perturbation which appears in the equations with the Neumann boundary conditions and the nonhomogeneous divergence condition. 
The iteration scheme then requires solving the wave equation using Dirichlet data which in turn is used to solve the linear Stokes system. 
In addition, the solution in each iteration step is used to prescribe new variable coefficients for the next step of the iteration.  The ellipticity of the variable coefficients is maintained by taking a sufficiently short time to ensure closeness of the inverse matrix of the flow map to its initial state of the identity and provide the necessary contraction-type estimates.  
We note that the global-in-time results for small data was established
in works~\cite{IKLT1,IKLT2,KO1,KO2,KT4}.

\colb
\startnewsection{The setting and the main result}{sec02}
We revisit the fluid-structure system coupling 
the incompressible Navier-Stokes equation and an
elasticity equation. The two equations are defined on domains
$\Omegaf(t)$ and $\Omegae(t)$ respectively, representing the
fluid and the solid regions, comprising a fixed bounded region
$\Omega= \Omegaf(t) \cup \Omegae(t)$ in $\mathbb{R}^{3}$.
The interaction is described by transmission boundary conditions on
the interface $\Gamma(t)$ between the two domains. The dynamics are
best described by the Lagrangian coordinates on the initial domain
configurations $\Omegaf(0)=\Omegaf$ and
$\Omegae(0)=\Omegae$. The incompressible Navier-Stokes
equations (\cite{Te1,Te2}) satisfied by the fluid velocity $v=(v_{1},v_{2},v_{3})$ and
the pressure $q$ are expressed in the  Lagrangian coordinates as
  \begin{align}
  &\partial_{t}v_{k} 
  -  a_{jl} \partial_{j}( a_{ml}
  \partial_{m} v_{k}) + a_{jk}\partial_{j}  q = 0
  \inon{in~$(0,T)\times\Omegaf $}
  \label{EQ01}
  \\&
  a_{ji}\partial_{j}v_{i}=0
  \inon{in~$(0,T)\times\Omegaf    $}
  ,
  \label{EQ02}
  \end{align}
for $k=1,2,3$, where we used the summation convention
on repeated indices. 
The symbol $a_{kj}$ denotes the $kj$ entry
of the $3 \times 3 $ matrix $a$ which is related to the flow map
$\eta$ by
\begin{align}
a= (\nabla \eta)^{-1}
.
\label{EQ03}
\end{align}
The flow map $\eta(t,\cdot)\colon \Omega \to \Omega$ is an unknown function
which determines the evolution of the domains. In particular,
we have
\begin{align}
\eta(t,\cdot)\colon\Omegaf \to \Omegaf(t)
\label{EQ132}
\end{align}
and
\begin{align}
\eta(t,\cdot)\colon \Omegae \to \Omegae(t)
,
\label{EQ133}
\end{align}
with the initial condition
\begin{align}
\eta(0,x)= x
\comma x\in\Omega_e
.
\label{EQ134}
\end{align}
On the other hand, the dynamics of the solid body and in particular
the displacement variable $w=(w_{1}, w_{2},w_{3})$, or $w(t,x)=
\eta(t,x)-x$, and its velocity $w_{t}$ are described by the
wave equation 
\begin{equation}
w_{tt}
- \Delta w
= 0
\inon{in~$(0,T)\times\Omegae $},
\label{EQ04}
\end{equation}
set in the natural Lagrangian variables.
Adjustments to more general stresses are possible (cf.~Remark~\ref{R01} below).

In this paper, we consider the case when the reference configuration $\Omega = \Omegaf \cup \Omegae \cup \Gammac$, $\Omegaf$, and $\Omegae$ are given by (see figure 1)
\begin{align}
\begin{split}
&
\Omega 
=
\{y=(y_1,y_2,y_3) \in \mathbb{R}^3 
:
(y_1, y_2) \in \mathbb{T}^2, 0<y_3< L_3\}
,
\\
&
\Omegaf
=
\{y=(y_1,y_2,y_3) \in \mathbb{R}^3 :
(y_1, y_2) \in \mathbb{T}^2, 0<y_3< L_1  ~\text{~or~}~ L_2<y_3< L_3 \},
\\&
\Omegae
=
\{y=(y_1,y_2,y_3) \in \mathbb{R}^3 :
(y_1, y_2) \in \mathbb{T}^2, L_1<y_3< L_2\}
,
\end{split}
\llabel{EQ313}
\end{align}
where $0<L_1 <L_2 <L_3$ and $\mathbb{T}^2$ is the two-dimensional
torus with the side~$2\pi$.
Thus, the common boundary to the fluid and the elastic body is expressed as
\begin{align}
\Gammac
=
\{(y_1, y_2) \in \mathbb{R}^2 :
(y_1, y_2, y_3) \in \Omega, ~y_3= L_1 ~\text{~or~}~ y_3= L_2\}
,
\llabel{EQ360}
\end{align}
while the outer boundary is represented by
\begin{align}
\Gammaf 
= 
\{y \in \bar{\Omega} : y_3 = 0 ~\text{~or~}~ y_3 = L_3\}
.
\llabel{EQ318}
\end{align}
The
interaction boundary conditions prescribed on the interface
$\Gammac(t)$ common to the two domains are the velocity and the stress
matching conditions, 
which are formulated in Lagrangian coordinates
over $\Gammac = \Gammac(0)$ as
\begin{align}
&v
= w_t
\inon{on~$(0,T)\times\Gammac$}
\label{EQ07}
\end{align}
and
\begin{align}
\frac{\partial w_k}{\partial N} = a_{jl} a_{ml} \partial_{m} v_{k} N^{j} - a_{jk} q N^{j}
\inon{on~$(0,T)\times\Gammac$},
\label{EQ135}
\end{align}
for $k=1,2,3$, where $N=(N^{1},N^{2},N^{3})$ denotes the 
unit normal at the interface $\Gammac$, which
is oriented outward
with respect to the domain $\Omegae$.

\tikzset{every picture/.style={line width=0.75pt}} 

\tikzset{every picture/.style={line width=0.75pt}} 
\begin{center}

\begin{tikzpicture}[x=0.75pt,y=0.75pt,yscale=-1,xscale=1]

\draw  [fill={rgb, 255:red, 255; green, 255; blue, 255 }  ,fill opacity=1 ] (106,172.8) -- (117.6,161.2) -- (226,161.2) -- (226,207.6) -- (214.4,219.2) -- (106,219.2) -- cycle ; \draw   (226,161.2) -- (214.4,172.8) -- (106,172.8) ; \draw   (214.4,172.8) -- (214.4,219.2) ;
\draw  [fill={rgb, 255:red, 255; green, 255; blue, 255 }  ,fill opacity=1 ] (106,126.4) -- (117.6,114.8) -- (226,114.8) -- (226,161.2) -- (214.4,172.8) -- (106,172.8) -- cycle ; \draw   (226,114.8) -- (214.4,126.4) -- (106,126.4) ; \draw   (214.4,126.4) -- (214.4,172.8) ;
\draw  [fill={rgb, 255:red, 255; green, 255; blue, 255 }  ,fill opacity=1 ] (106,80) -- (117.6,68.4) -- (226,68.4) -- (226,114.8) -- (214.4,126.4) -- (106,126.4) -- cycle ; \draw   (226,68.4) -- (214.4,80) -- (106,80) ; \draw   (214.4,80) -- (214.4,126.4) ;
\draw    (243,144) -- (353,144) ;
\draw [shift={(355,144)}, rotate = 180] [color={rgb, 255:red, 0; green, 0; blue, 0 }  ][line width=0.75]    (10.93,-3.29) .. controls (6.95,-1.4) and (3.31,-0.3) .. (0,0) .. controls (3.31,0.3) and (6.95,1.4) .. (10.93,3.29)   ;
\draw  [fill={rgb, 255:red, 255; green, 255; blue, 255 }  ,fill opacity=1 ] (364,79) -- (374.6,68.4) -- (484,68.4) -- (484,207.4) -- (473.4,218) -- (364,218) -- cycle ; \draw   (484,68.4) -- (473.4,79) -- (364,79) ; \draw   (473.4,79) -- (473.4,218) ;
\draw    (364,117) .. controls (396,89) and (436,150) .. (473.4,117) ;
\draw    (364,167) .. controls (403,190) and (434,144) .. (473.4,167) ;
\draw    (473.4,117) .. controls (477,107) and (482,114) .. (484,106) ;
\draw    (473.4,167) .. controls (481,150) and (477,181) .. (484,160) ;
\draw    (53,230) -- (53,176) ;
\draw [shift={(53,174)}, rotate = 90] [color={rgb, 255:red, 0; green, 0; blue, 0 }  ][line width=0.75]    (10.93,-3.29) .. controls (6.95,-1.4) and (3.31,-0.3) .. (0,0) .. controls (3.31,0.3) and (6.95,1.4) .. (10.93,3.29)   ;

\draw (251,123) node [anchor=north west][inner sep=0.75pt]   [align=left] {{\scriptsize under the flow map $\displaystyle \eta $}};
\draw (32,176) node [anchor=north west][inner sep=0.75pt]   [align=left] {$\displaystyle y_{3}$};
\draw (77,72) node [anchor=north west][inner sep=0.75pt]   [align=left] {$\displaystyle L_{3}$};
\draw (77,118) node [anchor=north west][inner sep=0.75pt]   [align=left] {$\displaystyle L_{2}$};
\draw (77,165) node [anchor=north west][inner sep=0.75pt]   [align=left] {$\displaystyle L_{1}$};
\draw (77,210) node [anchor=north west][inner sep=0.75pt]   [align=left] {$\displaystyle 0$};
\draw (148,96) node [anchor=north west][inner sep=0.75pt]   [align=left] {$\displaystyle \Omega _{f}$};
\draw (148,188) node [anchor=north west][inner sep=0.75pt]   [align=left] {$\displaystyle \Omega _{f}$};
\draw (148,144) node [anchor=north west][inner sep=0.75pt]   [align=left] {$\displaystyle \Omega _{e}$};
\draw (404,89) node [anchor=north west][inner sep=0.75pt]   [align=left] {$\displaystyle \Omega _{f}( t)$};
\draw (404,186) node [anchor=north west][inner sep=0.75pt]   [align=left] {$\displaystyle \Omega _{f}( t)$};
\draw (404,134) node [anchor=north west][inner sep=0.75pt]   [align=left] {$\displaystyle \Omega _{e}( t)$};
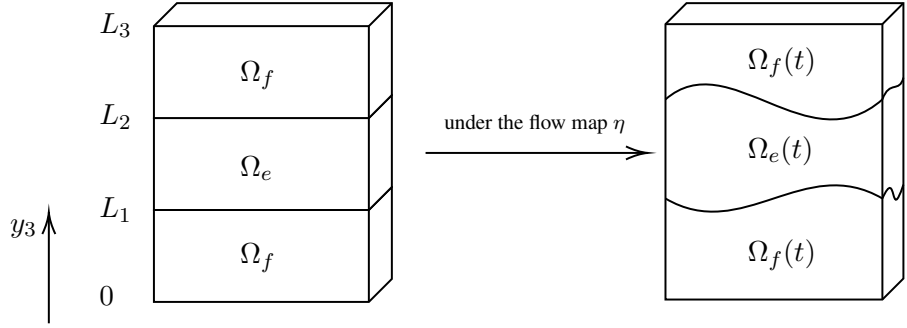
\draw (160,35) node [anchor=north west][inner sep=0.75pt]   [align=left] {Figure~1. Lagrangian domain to Eulerian domain};
\end{tikzpicture}
\end{center}
To close the system, we impose the usual no-slip boundary
condition 
\begin{align}
&v
= 0
\inon{on~$(0,T)\times\Gammaf $}
\label{EQ09}
\end{align}
on the outer boundary and the periodic boundary conditions for $v$, $q$, and $w$ in the horizontal directions, i.e.,
\begin{align}
v(t,\cdot),
q(t,\cdot),
w(t,\cdot)
~~\text{periodic~in~the~$y_1$~and~$y_2$~directions}.
\label{EQ08}
\end{align}
The initial conditions read
  \begin{align}
  \begin{split}
  & (v(0,\cdot), w(0,\cdot), w_{t}(0,\cdot))
  = (v_{0},w_0,w_1)
  \inon{in~ $ \Omegaf\times\Omegae\times\Omegae$}
  \\
  &
  v_0, w_0, w_1 ~~\text{periodic~in~the~$y_1$~and~$y_2$~directions}.
  \label{EQ10}
  \end{split}
  \end{align}
where $w_0=0$.
For $T>0$, we denote
\begin{align}
H^{r,s}((0,T) \times \Omegaf) = H^{r}((0,T)\scl L^{2}(\Omegaf)) \cap L^{2}((0,T)\scl H^{s}(\Omegaf))
\llabel{EQ05}
,
\end{align}
with the norm
\begin{equation}
\Vert f\Vert_{ H^{r,s}((0,T) \times \Omegaf)}^2
=
\Vert f\Vert_{H^{r}((0,T)\scl L^{2}(\Omegaf)) }^2
+
\Vert f\Vert_{L^{2}((0,T)\scl H^{s}(\Omegaf)) }^2
.
\llabel{EQ129}
\end{equation}
In Sections~\ref{sec05} and~\ref{sec06}, we shall work on a modified system with $T=1$, in order to avoid issue with constant dependence on small time.
For simplicity, we frequently write
$\Vert f\Vert_{H^{r,s}}=\Vert f\Vert_{ H^{r,s}((0,1) \times \Omegaf)}$.
We shall write $H_{\Gammac}^{r,s}$,
involving functions on $(0,1)\times \Gammac$,
for the analogous space corresponding to the boundary set~$\Gammac$.
It is also convenient to abbreviate
\begin{equation}
K^{s}=H^{s/2,s}
,
\llabel{EQ171}
\end{equation}
where $s\geq0$ and domain of integration is $(0,1)\times \Omegaf$ unless stated otherwise.
Similarly, we write
\begin{equation}
\Vert \cdot\Vert_{K^{s}} =
\Vert \cdot\Vert_{H^{s/2,s}}
.
\llabel{EQ173}
\end{equation}

Our main result is the local-in-time existence of strong solutions
to the system \eqref{EQ01}--\eqref{EQ04}, under the matching boundary conditions
\eqref{EQ07}--\eqref{EQ135}, boundary conditions \eqref{EQ09}--\eqref{EQ08},  
and the initial condition~\eqref{EQ10}
with minimal regularity on the initial data.

\cole
\begin{Theorem}
\label{T01} Let $s\in(3/2,3/2+\epsilon_0)$ for $\epsilon_0 \in
(0, 1/20)$. 
Assume that 
$v_{0} \in H^{s}(\Omegaf)$ with $\div v_0 =0$ 
and $w_{1} \in H^{s-1/2}(\Omegae)$ while 
$w_{0} = 0$
with the compatibility conditions
\begin{align} 
&w_{1} = v _{0}
\inon{on~$  \Gammac$}
\label{EQ31}
\\
&v_{0} = 0
\inon{on~ $  \Gammaf$}
\label{EQ32},  
\end{align}
and let $q_0$ be the initial pressure defined by
\begin{align}
\begin{split}
&
\Delta q_{0} = - \partial_{k} v_{0i} \partial_{i}  v_{0k}
\inon{in $\Omegaf$}
\\&
q_0 =    \left(\frac{\partial v_0}{\partial N} - \frac{\partial w_0}{\partial N}\right) \cdot N
\inon{on $\Gammac$}
\\&
\frac{\partial q_0}{\partial N}
= \Delta v \cdot N
\inon{on $\Gammaf$}
.
\end{split}
\llabel{EQ33}
\end{align}
Then system \eqref{EQ01}--\eqref{EQ04} with the coupling conditions \eqref{EQ07}--\eqref{EQ135}, boundary conditions \eqref{EQ09}--\eqref{EQ08}, and the initial condition \eqref{EQ10} admits a unique  solution
\begin{align}
\begin{split}
& v \in
K^{s+1}((0,T) \times \Omegaf) \\
& q \in
H^{s/2-1/2,s}((0,T) \times \Omegaf)  \\
&
\nabla q \in
K^{s-1}((0,T) \times \Omegaf)  \\
&
q|_{\Gammac} \in  K^{s-1/2}((0,T) \times \Gammac)
\\
& w \in C([0,T]\scl H^{s/2+5/4 }(\Omegae)) \\
& w_{t} \in   C([0,T]\scl H^{s/2 +1/4}(\Omegae)) \\
& \eta|_{\Omegaf} \in C^{1}([0,T]\scl  H^{s}(\Omegaf))
\end{split}
\llabel{EQ34}
\end{align}
for some constant $T>0$, where the corresponding norms are bounded by a function of the norms of the initial data.
\end{Theorem}
\colb
The proof of the theorem is given in Section~\ref{sec06} below.
As it is natural and common, we assume $w_0=0$, but all the statements apply, with only minor changes in the proof, if we assume more generally $w_0\in H^{s+1/2}(\Omegae)$ and $w_0|_{\Gammac}\in H^{s+1/2}(\Gammac)$, thus replacing \eqref{EQ134} with $\eta(0,x)=w_0(x)+x$.

\begin{Remark}
\rm
The range of $\epsilon_0$ is restricted to $(0,1/20)$  due to the quadratic relation from \eqref{EQ201} and Lemmas~\ref{L05} and~\ref{L07}.
More precisely, in order to apply the fixed-point argument in Section~\ref{sec06}, the size of the solution $v$ is bounded from below by a power of $\TT$ in \eqref{EQ500}.
On the other hand, it is bounded from above by another power of $\TT$ in Lemmas~\ref{L05} and~\ref{L07}.
%
Consequently, the range of $\epsilon_0$ is restricted in Theorem~\ref{T01}, in order to satisfy both bounds.
We emphasize that in the linear system of Section~\ref{sec05}, the range of $\epsilon_0$ can be made larger (see Remark~\ref{R02}). 
\end{Remark}

\begin{Remark}
\label{R01}
{\rm
In this paper, for simplicity of exposition, as in \cite{CS1}, we treat the Navier-Stokes equation with the stress tensor
$\mathbb{T}_1= \nabla u - p I$,
where $u$ and $p$ are Eulerian velocity and pressure respectively.
The proofs do not require much change if we consider instead the stress
$\mathbb{T}_2= \nabla u + (\nabla u)^{T}- p I$.
The two stresses lead to the same form of the Navier-Stokes system since
$\div {\mathbb T}_1=\div {\mathbb T}_2=\Delta u - \nabla p$, but they lead to the
different form of \eqref{EQ135}, which for the stress ${\mathbb T}_2$
reads
\begin{equation}
\frac{\partial w_k}{\partial N} 
= 
a_{jl} a_{ml} \partial_{m} v_{k} N^{j} 
+   a_{jl} a_{mk} \partial_{m} v_{l} N^{j}
- a_{jk} q N^{j}
,
\llabel{EQ170}
\end{equation}
where $k=1,2,3$.
The theorem and the proof can easily be adapted to this different stress.
The only important difference is that the proof of the Stokes theorem in~\cite{KLT2} requires that we use
the stress~$T_1$ rather than~$T_3$ from~\cite{GS}.
}
\end{Remark}

The following proposition shows the main difficulty
when applying Poincar\'e-type inequalities in fractional spaces on small time intervals.

\begin{Proposition}
	\label{P01}
Let $0<\alpha<1/2$ and $\beta\in \mathbb{R}$. 
There exists a constant $C\geq 1$ such that
\begin{align}
	\Vert f\Vert_{L^2 (0,T)}
	\leq
	C T^\beta
	( \lfloor f \rfloor_{H^\alpha (0,T)} 
	+
	 \Vert f \Vert_{L^2 (0,T)} 
	)
        \comma f \in C^\infty ([0,T])
	\comma f(0)=0
	\label{EQ887}
\end{align}
for all $T\in(0,1]$
if and only if $\beta \leq 0$.
\end{Proposition}

Above, we refer to the usual 
Sobolev-Slobodeckij definition
of the fractional norm
\begin{align}
	\lfloor f \rfloor_{H^\alpha (0,T)} 
	:
	=
	\left(
	\iint_{(0,T)^2} 
	\frac{|f(t_1) - f(t_2)|^2}{|t_1 -t_2|^{2\alpha+1}}
	\, dt_1 dt_2
	\right)^{\frac{1}{2}},
   \label{EQ85}
\end{align}
for $T>0$ and $\alpha \in (0,1)$.


\enew
%
%
%
\colb

\begin{proof}[Proof of Proposition~\ref{P01}]
Since the sufficiency is clear, we only prove the necessity.
With $M \geq 2$,
let $f(t)$ be a smooth function with $f(0) = 0$,  $f(t) = 1$ on $[T/M, T]$, and $0\leq f(t) \leq 1$ on $[0,T]$ such that $\Vert f'(t)\Vert_{L^\infty (0,T)} \leq \tilde{C} M/T$, where $\tilde{C}>0$ is a constant.
Note that
  \begin{equation}
   C^{-1} T^{1/2}\leq \Vert f\Vert_{L^2 (0,T)}\leq T^{1/2}
   .
   \label{EQ89}
  \end{equation}
On the other hand, from \eqref{EQ85} it follows that
\begin{align}
	\lfloor
	f
	\rfloor_{H^\alpha (0,T)}
	\leq
	C_\alpha T^{1/2-\alpha} M^{\alpha -1/2},
   \label{EQ87}
\end{align}
where $C_\alpha>0$ is a constant depending on $\alpha$.
Using \eqref{EQ89} and \eqref{EQ87} in \eqref{EQ887}, we get
\begin{align*}
	C^{-1}T^{1/2}\leq C_\alpha T^{\beta}(T^{1/2-\alpha}M^{\alpha-1/2}+ T^{1/2}).
\end{align*}
Taking $M$ sufficiently large, we obtain $T^{\beta}\geq 1/C$, leading to $\beta\leq 0$.
\end{proof}
\colb

In order to solve the linear Stokes-wave system in Section~\ref{sec05}, we consider the linear Stokes system with a nonhomogeneous divergence
\begin{align}
\begin{split}
&u_{t} -  \Delta u + \nabla  p = f
\inon{in~$ (0,T)\times \Omegaf  $}
\\
&
\div u = g
\inon{in~$ (0,T)\times\Omegaf    $}
.
\end{split}
\label{EQ15}
\end{align}
The system is supplemented with two boundary conditions and the initial data
\begin{align}
&\frac{\partial u}{\partial N}
- p N
= 
\hff
\inon{on~ $ (0,T)\times\Gammac $}
\label{EQ06}
\\&
u = \hf
\inon{on~$ (0,T)\times\Gammaf$}
\label{EQ23}
\\&u(0,\cdot)
=
u_0
\inon{in~ $\Omegaf$}
\label{EQ130}
\end{align}
We also impose the compatibility condition
\begin{align}
u_0|_{\Gamma_2}=\hf|_{t=0}
.
\label{EQ61}
\end{align}

We recall the following maximal regularity theorem for the Neumann-Stokes problem.

\cole
\begin{Lemma}[\cite{KLT2}]
\label{T02} 
Let $s\in [1,3/2)\cup (3/2, 2)$. 
Assume that
\begin{align}
u_{0} \in H^{s}(\Omegaf ) 
\llabel{EQ156}
\end{align}
and
\begin{align}
\begin{split}
& 
(f,g,\hff,\hf)
\in
(
K^{s-1},
K^{s}, 
H^{s/2-1/4,s-1/2}_{\Gammac},
H^{s/2+1/4,s+1/2}_{\Gammaf}
)
,
\end{split}
\llabel{EQ35}
\end{align}
with the structural condition
\begin{equation}
g_{t} 
= 
\tilde{g}+ \div b
,
\llabel{EQ148}
\end{equation}
where
\begin{equation}
\tilde{g}, b  \in  K^{s-1} ((0,T)\times \Omegaf)
.
\llabel{EQ149}
\end{equation}
Then there exists a unique solution $(u,p)$ to the system 
\eqref{EQ15}--\eqref{EQ61}
satisfying 
\begin{align}
\begin{split}
&
\Vert u \Vert_{K^{ s+1} ((0,T)\times \Omegaf)} 
+ 
\Vert p  \Vert_{H^{s/2-1/2, s} ((0,T)\times \Omegaf)}
+ 
\Vert \nabla p  \Vert_{K^{s-1} ((0,T)\times \Omegaf)}
+ 
\Vert p  \Vert_{H^{s/2-1/4, s-1/2} ((0,T)\times \Gammac)}
\\&\indeq
\les
\Vert u_{0} \Vert_{H^{s}} 
+  
\Vert f \Vert_{K^{s-1} ((0,T)\times \Omegaf)} 
+ 
\Vert g \Vert_{K^{s} ((0,T)\times \Omegaf)} 
+ 
\Vert \tilde{g} \Vert_{K^{ s-1} ((0,T)\times \Omegaf)} 
+ 
\Vert b \Vert_{K^{ s-1} ((0,T)\times \Omegaf)} 
\\&\indeq\indeq
+ 
\Vert \hff \Vert_{H^{s/2-1/4, s-1/2} ((0,T)\times \Gammac)}
+ 
\Vert \hf \Vert_{H^{s/2+1/4, s+1/2} ((0,T)\times \Gammaf)}
,
\llabel{EQ17}
\end{split}
\end{align}
where the implicit constant depends on~$\Omegaf$ and $T$.
\end{Lemma}
\colb

See \cite{MZ1,MZ2} for a related maximal regularity theorem.

\startnewsection{Space-time trace, interpolation, and hidden regularity inequalities}{sec03}
In this section, we recall four auxiliary results needed in the fixed-point arguments.
The first lemma provides an estimate for the trace in a space-time norm and is an essential ingredient 
when solving the linear Stokes problem with nonhomogeneous divergence and 
when constructing
solutions to the linear Stokes-wave system in the next section.

\cole
\begin{Lemma}[\cite{KLT1}]
\label{L01}
Let $r>1/2$, $\theta\geq0$, and $T>0$.
If 
$u\in 
L^{2}([0,T]\scl H^{r}(\Omegaf))
\cap
H^{2\theta r/(2r-1)} ([0,T], L^{2}(\Omegaf))$, then 
$
u \in H^{\theta} ([0,T], L^{2}(\Gammac))
$, and for all $\epsilon \in(0,1]$,
we have the inequality
\begin{align}
\begin{split}
\Vert u\Vert_{H^{\theta}([0,T]\scl  L^{2}(\Gammac))}
\leq
\epsilon
\Vert u \Vert_{H^{2\theta r/(2r-1)}([0,T]\scl  L^{2}(\Omegaf))}
+ 
C\epsilon^{1-2r}
\Vert u \Vert_{L^{2}([0,T]\scl  H^{r}(\Omegaf))} 
,
\end{split}
\llabel{EQ65}
\end{align}
where $C>0$ is a constant, which depends on $\Omegaf$ and~$T$.
\end{Lemma}
\colb

The second lemma provides a space-time interpolation inequality needed in several places in the next two sections.

\cole
\begin{Lemma}[\cite{KLT1}]
\label{L02}
Let $\alpha,\beta>0$ and~$T>0$.
If 
$u\in H^{\alpha}([0,T]\scl L^2(\Omegaf))
\cap
L^2([0, T], H^{\beta}(\Omegaf))$, then we have that
$
u \in H^{\theta} ([0, T], H^{\lambda}(\Omegaf))
$ for all $\theta\in(0,\alpha)$ and $\lambda\in(0,\beta)$ such that
\begin{equation}
\frac{\theta}{\alpha} + \frac{\lambda}{\beta} \leq 1
.
\llabel{EQ66}
\end{equation}
In addition, for all $\epsilon \in (0,1]$,
we have the inequality
\begin{align}
\begin{split}
\Vert u\Vert_{H^{\theta} ([0,T], H^{\lambda}(\Omegaf))}
\leq
\epsilon
\Vert u \Vert_{H^{\alpha}([0,T]\scl L^2(\Omegaf))}
+ 
C \epsilon^{-\frac{\theta \beta}{\alpha\lambda}}
\Vert u \Vert_{L^2([0, T], H^{\beta}(\Omegaf))} 
,
\end{split}
\llabel{EQ70}
\end{align}
where $C>0$ is a constant depending on $\Omegaf$ and~$T$.
\end{Lemma}
\colb

Next we recall the hidden regularity result for the wave equation 
\begin{align}
\begin{split}
&
w_{tt}   - \Delta{w}= 0
\inon{in~ $ (0,T)\times\Omegae $}
\\&
w = \psi    
\inon{on~$  (0,T)\times\partial\Omegae $}
\\&
\text{$w$~periodic~in~the~$y_1$~and~$y_2$~directions}
,
\end{split}
\label{EQ14}
\end{align}
with the initial data 
\begin{equation}
(w,w_t)(0)= (w_0,w_1) 
\inon{in~ $\Omegae $}  
\label{EQ158}
\end{equation}
from~\cite{LLT}.

\cole
\begin{Lemma}[\cite{LLT}]
\label{L03}
Assume
that
$(w_{0},w_{1}) \in H^{\beta}( \Omegae)\times H^{\beta -1}( \Omegae)$,
where $\beta \geq1$,
and 
$$\psi \in C([0,T]\scl  H^{\beta-1/2}(\partial \Omegae))    \cap H^{\beta, \beta}((0,T)\times \partial \Omegae),$$
with the compatibility condition $\psi|_{t=0} = w_0 |_{\Gammac}$ and $\pt \psi |_{t=0} = w_1|_{\Gammac}$.
Then there exists a solution $w$ 
of \eqref{EQ14}--\eqref{EQ158}, which satisfies the estimate
\begin{align}
\begin{split}
&\Vert w \Vert_{C([0,T]\scl   H^{\beta}( \Omegae))} 
+ \Vert w_{t} \Vert_{C([0,T]\scl   H^{\beta-1}( \Omegae))}
+ \left\Vert 
\frac{\partial w}{\partial N}
\right\Vert_{H^{\beta-1, \beta -1}((0,T) \times \partial \Omegae)} 
\\&\indeq
\lec
\Vert w_{0} \Vert_{H^{\beta}( \Omegae)} 
+
\Vert w_{1} \Vert_{H^{\beta-1}( \Omegae)} 
+
\Vert \psi \Vert_{H^{\beta, \beta}((0,T) \times \partial \Omegae)}
,
\end{split}
\llabel{EQ139}
\end{align}
where the implicit constant depends on $\Omegae$ and $T$.
\end{Lemma}
\colb

At last, we state an essential 
trace regularity result for the wave equation from \cite{RV}.

\cole
\begin{Lemma}
[\cite{RV}]
\label{L04}
Assume that
$(w_{0},w_{1}) \in H^{\beta+2}( \Omegae)\times H^{\beta +1}( \Omegae)$,
where $0< \beta<5/2$,
and 
$$\psi \in L^2 ((0,T)\scl  H^{\beta +2} ( \partial \Omegae))    \cap H^{\beta/2 +1}((0,T)\scl H^{\beta/2 +1}( \partial \Omegae)),$$
with the compatibility condition $\pt \psi |_{t=0} = w_1 |_{\Gammac}$.
Then there exists a solution $w$
of \eqref{EQ14}--\eqref{EQ158} such that
\begin{align}
\begin{split}
\left\Vert 
\frac{\partial w}{\partial N}
\right\Vert_{L^2 ( (0,T), H^{\beta +1} ( \partial \Omegae))} 
&\lec
\Vert w_{0} \Vert_{H^{\beta +2}( \Omegae)} 
+ 
\Vert w_{1} \Vert_{H^{\beta +1}( \Omegae)} 
+
\Vert \psi \Vert_{L^2 ((0,T)\scl H^{\beta +2}( \partial \Omegae))} 
\\&\indeq 
+ 
\Vert \psi \Vert_{H^{\beta/2 +1}((0,T)\scl H^{\beta/2 +1}( \partial \Omegae))}
,
\end{split}
\llabel{EQ139}
\end{align}
where the implicit constant depends on $\Omegae$ and $T$.
\end{Lemma}
\colb

\startnewsection{Solutions to a linear Stokes-wave system}{sec05}
When solving the nonlinear problem, we are faced with
the following modified linear problem.

Let $\TT \in (0,1/4]$, to be specified below,
and let $\psi_{\TT} (t)$ be a smooth cutoff function with
\begin{equation*}
\psi_{\TT} (t)=
\begin{cases}
1 \inon{on} ~~~[0,\TT]
\\
0 \inon{on} ~~~[2\TT, 1],
\end{cases}
\end{equation*}
such that $\Vert \psi_{\TT}' (t) \Vert_{L^\infty (0,1)} \lec 1/\TT$.
Consider the problem
\begin{align}
&v_{t} -  \Delta v+ \nabla  q = f
\inon{in~$ (0,1)\times\Omegaf    $}
\label{EQ218}
\\&
\div v = g
\inon{in $(0,1)\times\Omegaf     $}
\label{EQ280}
\\&
w_{tt} - \Delta w=0    \inon{in~$ (0,1)\times\Omegae    $}
\label{EQ219}
\end{align}
and
\begin{align}
&
w(t,x) 
= 
w_0(x)
+
\int_0^t
\psi_{\TT}(\tau)
v(\tau,x) \,d\tau
\inon{on  $ [0,1)\times\Gammac    $}
\label{EQ320}
\\&
\frac{\partial v}{\partial N}- q N = \dwdn + h
\inon{on  $[0,1)\times\Gammac    $}
\label{EQ321}
\\&
v =0    \inon{on  $ [0,1)\times\Gammaf    $}.
\label{EQ322}
\end{align}
Note that differentiating in time of \eqref{EQ320} leads to
\begin{align}
w_t(t,x) = \psi_{\TT} (t) v(t,x)
\inon{on $[0,1) \times \Gac$}
,
\label{EQ007}
\end{align}
which is the modified version of the law $w_t=v$ on $\Gammac$.
The next statement provides the existence of solutions for the system \eqref{EQ218}--\eqref{EQ322}.

\cole
\begin{Lemma}
\label{L10} Let $s\in(3/2,3/2+\epsilon_0)$ where $\epsilon_0 \in (0, 1/20)$.
Suppose that the initial data satisfies
\begin{align}
\begin{split}
& 
(v_{0},w_{0},w_{1}) \in H^{s}(\Omegaf) 
\times
H^{s+1/2  }(\Omegae)
\times
H^{s-1/2  }(\Omegae)
\end{split}
\llabel{EQ36}
\end{align}
such that $\div v_0 =0$ 
and with the compatibility conditions \eqref{EQ31}--\eqref{EQ32} and
\begin{align*}
\frac{\partial v_0}{\partial N}
- 
q_0 N 
= 
\frac{\partial w_0}{\partial N} 
+ 
h(0)
\inon{on  $\Gammac    $}.
\end{align*}
Assume that the nonhomogeneous terms satisfy
  \begin{align}   
  \begin{split}
   (f,g,h)
   \in
   K^{s-1}  ((0,1) \times \Omegaf) \times K^{s} ((0,1) \times \Omegaf) \times H^{s/2-1/4,s-1/2} ((0,1) \times \Gammac)
   ,
  \end{split}
   \llabel{EQ37}
  \end{align}
where
  \begin{equation}
    g_{t} = \AA+ \div \BB 
    ,   
   \llabel{EQ76}
  \end{equation}
with $\AA,\BB  \in  K^{s-1}((0,1) \times \Omegaf)$.
There exists a sufficiently small constant
$T_0$ such that if $\TT$ satisfies $\TT\in (0,T_0]$, then the system
\eqref{EQ218}--\eqref{EQ322} admits a unique solution 
  \begin{align}
  \begin{split}
    &v \in  K^{s+1}((0,1) \times \Omegaf) 
     \\&
    q \in  H^{s/2-1/2, s}((0,1) \times \Omegaf)
     \\&
     \nabla q \in  K^{s-1} ((0,1) \times \Omegaf)
     \\&
     q|_{\Gammac} \in  H^{s/2-1/4, s-1/2}((0,1) \times \Gammac)
     \\&
    w \in C([0,1], H^{s/2+5/4  }(\Omegae))
     \\&
    w_{t} \in C([0,1], H^{s/2 +1/4 }(\Omegae))
  \end{split}
   \llabel{EQ38}
   \end{align}
for which 
\begin{align}
\begin{split}
&
\Vert v \Vert_{K^{ s+1}  }
+ 
\Vert q  \Vert_{H^{s/2-1/2, s} }
+ 
\Vert \nabla q  \Vert_{K^{s-1} }
+ \Vert q  \Vert_{H^{s/2-1/4, s-1/2}_{\Gammac}}
\\&\indeq
\leq
\CCC
(\TT^{-1/7}
\Vert v_{0} \Vert_{H^{s}} 
+ \Vert w_{0} \Vert_{H^{s+1/2}} 
+ \Vert w_{1} \Vert_{H^{s-1/2}}
\\&\indeq\indeq
+  \Vert f \Vert_{K^{ s-1}}
+ \Vert g      \Vert_{K^{ s}}
+ \Vert \AA \Vert_{K^{ s-1}} 
+ \Vert \BB  \Vert_{K^{ s-1}} 
+      \Vert h \Vert_{H^{s/2-1/4, s-1/2}_{\Gammac}})
,
\end{split}
\label{EQ90}
\end{align}
where $\CCC>0$ is a constant.
Also, the elastic displacement satisfies
  \begin{align}
  \begin{split}
   &
   \Vert w \Vert_{C([0,1]\scl  H^{s/2+5/4}(\Omegae))} 
   + \Vert w_{t} \Vert_{C([0,1]\scl  H^{s/2+1/4}(\Omegae))}
   \\&\indeq
   \lec_{\TT}
   \Bigl(
     \Vert v_{0} \Vert_{H^{s}} 
     + \Vert w_{0} \Vert_{H^{s+1/2}} 
     + \Vert w_{1} \Vert_{H^{s-1/2}}
  \\&\indeq\indeq\indeq\indeq\indeq\indeq
     +  \Vert f \Vert_{K^{ s-1}}
     + \Vert g      \Vert_{K^{ s}}
     + \Vert \AA \Vert_{K^{ s-1}} 
     + \Vert \BB  \Vert_{K^{ s-1}} 
     +      \Vert h \Vert_{H^{s/2-1/4, s-1/2}_{\Gammac}}
   \Bigr)
   .
  \end{split}
   \label{EQ43}
  \end{align}

\end{Lemma}
\colb

In \eqref{EQ90}, as well as in the rest of the paper,
we do not indicate the domains $(0,1) \times \Omegaf$ or $(0,1) \times \Omegae$ in the norms as they are understood from the context,
However, we always use a complete notation for norms involving the boundary trace.

\begin{Remark}
\label{R02}
\rm
In Lemma~\ref{L10}, the range of $\epsilon_0$ can be relaxed to $(0, \sqrt{2} - 1)$.
Indeed, the only two modifications needed are the power $1/6$
of~$\TT$ in \eqref{EQ365},
which can be replaced by an arbitrarily small $\epsilon>0$, 
and the power $-1/7$ of~$\TT$ in \eqref{EQ201}, which can also be replaced
with a smaller number $-2/5$.
The estimates \eqref{EQ90} and \eqref{EQ43} then still hold by replacing the power $-1/7$ of $\TT$ in \eqref{EQ90} by $-2/5$ and taking
the implicit constants sufficiently large.
\end{Remark}

\begin{proof}[Proof of Lemma~\ref{L10}]
In order to apply the fixed point argument, we consider the inequality
\begin{align}
	\begin{split}
	\Vert v\Vert_{K^{s+1}}
	&
	\leq
	\CCC
	(
	\TT^{-1/7}
	\Vert v_{0} \Vert_{H^{s}} 
	+ \Vert w_{0} \Vert_{H^{s+1/2}} 
	+ \Vert w_{1} \Vert_{H^{s-1/2}}
	\\&\indeq\indeq
	+  \Vert f \Vert_{K^{ s-1}}
	+ \Vert g      \Vert_{K^{ s}}
	+ \Vert \AA \Vert_{K^{ s-1}} 
	+ \Vert \BB  \Vert_{K^{ s-1}} 
	+      \Vert h \Vert_{H^{s/2-1/4, s-1/2}_{\Gammac}}
	)
	,
	\label{EQ500}
\end{split}
\end{align}
where $\CCC\geq 1$ is a sufficiently large constant to be determined below.
Denote
\begin{align}
\begin{split}
\mathcal{Z}
&
= 
\Bigl\{v \in K^{s+1} ((0,1) \times \Omegaf):
v(0) = v_0 ~~\text{in}~~\Omegaf, 
v=0 ~~\text{on}~~ (0,1) \times\Gammaf, 
\\&\indeq\indeq\indeq
v~~\text{periodic in the $y_1$~and~$y_2$ directions},
~~\text{and}~~\text{\eqref{EQ500}~holds}\Bigr\}
.
\label{EQ388}
\end{split}
\end{align}
\colb
For $v\in \mathcal{Z}$,
we write
  \begin{equation}
   \tLambda v = \bar w
   ,
   \label{EQ27}
  \end{equation}
where $\bar w$
solves the wave equation
 \begin{align}
\begin{split}
\bar{w}_{tt} - \Delta \bar{w}
=
0    \inon{in $(0,1)\times\Omegae   $}
,
\end{split}
\label{EQ385}
\end{align}
with the boundary and initial conditions
\begin{align}
&
\bar{w}(t, x)
= w_0 (x)
+ 
\int_{0}^{t} \psi_{\TT} (\tau) v (\tau, x)
\,d\tau
\inon{on $(0,1)\times\Gamma_c$}
\label{EQ387}
\\&
\text{$\bar{w}$~periodic~in~the~$y_1$~and~$y_2$ ~directions}
\label{EQ488}
\\&
(\bar{w},\bar{w}_t)(0)=(w_0,w_1)
\inon{in~$\Omegae$}
\label{EQ489}.
\end{align}
Then, we define a mapping
\begin{align}
\Lambda \colon v (\in \mathcal{Z} )\mapsto \bar{v}
,
\llabel{EQ152}
\end{align}
where $(\bar{v}, \bar{q})$ is the solution of the linear Stokes equations
\begin{align}
&\bar{v}_t 
-  
\Delta \bar{v}
+ 
\nabla  \bar{q} = f
\inon{in $(0,1)\times \Omegaf$}
\label{EQ83}
\\&
\div \bar{v} 
= 
g
   \inon{in $(0,1)\times \Omegaf$}
\label{EQ84}
\end{align}
with the boundary and initial conditions 
\begin{align}
&
\frac{\partial \bar{v}}{\partial N}
- 
\bar{q} N 
= 
\frac{\partial \bar{w}}{\partial N} 
+ 
h
\inon{on  $(0,1)\times\Gammac $}
\label{EQ86}
\\&
\bar{v}
=
0    \inon{on  $ (0,T)\times\Gammaf    $}
\\
&
\text{$\bar{v}$~periodic~in~the~$y_1$~and~$y_2$ ~directions}
\label{EQ190}
\\&
\bar{v} (0)=v_0 
\inon{in~$\Omegaf$}.
\label{EQ191}
\end{align}
We shall prove that the mapping $\Lambda$ is well-defined from $\mathcal{Z}$ to $\mathcal{Z}$,
for some sufficiently large constant $\bar{C}\geq 1$ in \eqref{EQ500} and sufficiently small constant $\TT\in (0,1/4]$.

From Lemma~\ref{T02}, the system \eqref{EQ83}--\eqref{EQ84} with the boundary and initial conditions \eqref{EQ86}--\eqref{EQ191} admits a unique solution $(\bar v,\bar q)$ which satisfies
\begin{align}
\begin{split}
&
\Vert \bar v \Vert_{K^{ s+1}}
+ 
\Vert \bar q \Vert_{H^{s/2-1/2, s}}
+ 
\Vert \nabla\bar q \Vert_{K^{s-1}}
+ 
\Vert \bar{q}  \Vert_{H^{s/2-1/4, s-1/2}_{\Gammac}}
\\&
\lec 
\Vert  v_{0} \Vert_{H^{s}} 
+  \Vert f \Vert_{K^{ s-1}}
+ \Vert g \Vert_{K^{ s}}
+ \Vert \AA \Vert_{K^{ s-1}} 
+ \Vert \BB  \Vert_{K^{ s-1}} 
+ 
\left\Vert h
\right\Vert_{H^{s/2-1/4, s-1/2}_{\Gammac}}
+ 
\left\Vert
\dwbardn 
\right\Vert_{H^{s/2-1/4, s-1/2}_{\Gammac}}
.
\end{split}
\label{EQ29}
\end{align}
It remains to estimate the last term on the right side of \eqref{EQ29}.
For the space part of the norm, we invoke the trace regularity result Lemma~\ref{L04}, obtaining
  \begin{align}
\begin{split}
\left\Vert \frac{\partial \bar w}{\partial N} \right
\Vert_{L_t^2  H_x^{s-1/2} (\Gammac)} 
&
\lec
\Vert w_{0} \Vert_{H^{s+1/2} } 
+ 
\Vert
w_{1} \Vert_{H^{s-1/2} } 
+  
\Vert  \bar w  \Vert_{L_t^2  H_x^{s+1/2}(\Gammac)}
+ 
\Vert  \bar w \Vert_{H_t^{s/2+1/4}  H_x^{s/2+1/4} (\Gammac)}
\\    &
\lec
\Vert w_{0} \Vert_{H^{s+1/2}} 
+ 
\Vert
w_{1} \Vert_{H^{s-1/2} } 
+ 
\Vert  \bar{w} \Vert_{L^2_t  H_x^{s+1/2} (\Gammac)}
+  
\Vert  \bar{w}_t  \Vert_{H_t^{s/2-3/4}  H_x^{s/2+1/4}(\Gammac)}
.
\end{split}
\label{EQ28}
\end{align}
For the third term on the far right side of \eqref{EQ28}, we appeal to the boundary condition \eqref{EQ387} and the trace inequality to get
\begin{align}
\begin{split}
&
\Vert  \bar w \Vert_{L^2_t H_x^{s+1/2}(\Gammac)}^2
\lec
\left\Vert
\int_{0}^{1} \psi_{\TT} v \, d\tau
\right\Vert_{H^{s+1/2}(\Gammac)}^2
\lec
\TT
\int_{0}^{2\TT} \Vert v  \Vert^{2}_{H^{s+1}}\,d\tau
\lec
\TT \Vert v\Vert_{K^{s+1}}^2
,
\end{split}
\label{EQ39}
\end{align}
since $w_0=0$.
For the last term on the far right side of \eqref{EQ28}, we deduce that
\begin{align}
\begin{split}
\Vert \bar{w}_t\Vert_{H^{s/2-3/4}_t H^{s/2+1/4}_x (\Gammac)}
&\lec
{\Vert \bar{w}_t\Vert_{H^{1}_t H^{s-3/2}_x (\Gammac)}^{s/2-3/4}}
{\Vert \bar{w}_t\Vert_{L^2_t H_x^{(-4s^2+16s-7)/2(7-2s)} (\Gammac)}^{7/4-s/2}}
              \\&
      = \mathcal{I}_1^{s/2-3/4}
        \mathcal{I}_2^{7/4-s/2}
,
\label{EQ599}
\end{split}
\end{align}
which follows from the Sobolev interpolation. 
The factor $\mathcal{I}_1$ is estimated using \eqref{EQ007} and the trace inequality as
\begin{align}
\begin{split}
\mathcal{I}_1
&
\lec
\Vert \psi_{\TT}' v\Vert_{L^2_t H^{s-1}_x}
+
\Vert \psi_{\TT} v_t \Vert_{L^2_t H^{s-1}_x}
+
\Vert \psi_{\TT} v \Vert_{L^2_t H^{s-1}_x}
\\&
\lec
\TT^{-1}
\Vert v\Vert_{L^2_t H^{s-1}_x ((0,2\TT)\times \Omegaf)}
+
\Vert v\Vert_{K^{s+1}}
\lec
\TT^{-s/(s+1)} \Vert v\Vert_{K^{s+1}}
;
\label{EQ351}
\end{split}
\end{align}
in the last step of \eqref{EQ351}, we have used
\begin{align*}
\begin{split}
\Vert v\Vert_{L^2_t H^{s-1}_x ((0,2\TT) \times \Omegaf)}^2
&
\lec
\int_0^{2\TT}
\Vert v\Vert_{H^{s-1}}^2 \, d\tau
\lec
\int_0^{2\TT}
\Vert v\Vert_{L^{2}}^{4/(s+1)}
\Vert
v\Vert_{H^{s+1}}^{2(s-1)/(s+1)}
 \, d\tau
 \\&
 \lec
\Vert v\Vert_{L^\infty_t L_x^{2}}^{4/(s+1)}
\TT^{2/(s+1)}
\Vert v\Vert_{L^2_t H^{s+1}_x}^{2(s-1)/(s+1)}
\lec
\TT^{2/(s+1)}
\Vert v\Vert_{K^{s+1}}^2 
,
\end{split}
\end{align*}
which follows from Lemma~\ref{L02} and H\"older's inequality.
Similarly, we bound the term $\mathcal{I}_2$ as
\begin{align}
\begin{split}
\mathcal{I}_2
&
\lec
\Vert \psi_{\TT} 
v\Vert_{L^2_t H_x^{s} }
\lec
\TT^{1/2(s+1)}
\Vert v\Vert_{K^{s+1}}
.
\end{split}
\label{EQ153}
\end{align}
since $(-4s^2+16s-7)/(14-4s)+1/2=s$.
From \eqref{EQ599}--\eqref{EQ153} we infer that
\begin{align}
\begin{split}
\Vert \bar{w}_t\Vert_{H^{s/2-3/4}_t H^{s/2+1/4}_x (\Gammac)}
&
\lec
\TT^{(-4s^2+4s+7)/8(s+1)}
\Vert v\Vert_{K^{s+1}}
\lec
\TT^{1/6} \Vert v\Vert_{K^{s+1}}
,
\label{EQ365}
\end{split}
\end{align}
since $s\in (3/2, 31/20)$.
Inserting \eqref{EQ39} and \eqref{EQ365} to \eqref{EQ28}, we conclude that
\begin{align}
\begin{split} 
\left\Vert \frac{\partial \bar w}{\partial N} \right\Vert_{L^2_t H^{s-1/2}_x (\Gammac)}
&
\lec
\Vert w_{0} \Vert_{H^{s+1/2}} 
+ 
\Vert w_{1} \Vert_{H^{s-1/2}} 
+     
\TT^{1/6}
\Vert v\Vert_{K^{s+1}}
.
\end{split}
\label{EQ41}
\end{align}             
For the time part of the last term on the far right side of \eqref{EQ29}, we use the hidden regularity Lemma~\ref{L03} to obtain
\begin{align}
\begin{split}
\left\Vert \frac{\partial \bar{w}}{\partial N}\right\Vert_{H^{s/2-1/4}_t L_x^2(\Gammac)}
&
\lec
\Vert w_0\Vert_{H^{s/2+3/4}}
+
\Vert w_1\Vert_{H^{s/2-1/4} }
+
\Vert \bar{w}\Vert_{H^{s/2+3/4}_t L^2_x (\Gammac)}
+
\Vert \bar{w}\Vert_{ L^2_t H^{s/2+3/4}_x (\Gammac)}
.
\label{EQ353}
\end{split}
\end{align}
The third term on the right side is estimated as
\begin{align}
\begin{split}
\Vert \bar{w}\Vert_{H^{s/2+3/4}_t L^2_x (\Gammac)}
&
\lec
\Vert \bar{w}_t\Vert_{H^{s/2-1/4}_t L^2_x (\Gammac)}
+
\Vert \bar{w}\Vert_{L^{2}_t L^2_x (\Gammac)}
\\&
\lec
\Vert \bar{w}_t\Vert_{H^{1}_t L^2_x (\Gammac)}^{s/2-1/4}
\Vert  \bar{w}_t\Vert_{L^2_t L^2_x (\Gammac)}^{5/4-s/2}
+
\Vert \bar{w} \Vert_{L^2_t H^{s+1/2}_x (\Gammac)}
\\&
=
{\mathcal{I}_3}^{s/2-1/4}
{\mathcal{I}_4}^{5/4-s/2}
+
\Vert \bar{w} \Vert_{L^2_t H^{s+1/2}_x (\Gammac)}
,
\label{EQ970}
\end{split}
\end{align}
where we used the Sobolev interpolation.
The term $\mathcal{I}_3$ is bounded using the trace inequality as
\begin{align}
\begin{split}
\mathcal{I}_3
&
\lec
\Vert \psi_{\TT}' v \Vert_{L^2_t L^2_x (\Gammac)}
+
\Vert \psi_{\TT} v_t \Vert_{L^2_t L^2_x (\Gammac)}
+
\Vert \psi_{\TT} v \Vert_{L^2_t L^2_x (\Gammac)}
\\&
\lec
\Vert \psi'_{\TT} (v-v_0) \Vert_{L^2_t H^{s-1}_x ( (0,2\TT)\times \Omegaf)}
+
\Vert \psi'_{\TT} v_0 \Vert_{L^2_t H^{s-1}_x ( (0,2\TT)\times \Omegaf)}
+
\Vert v \Vert_{K^{s+1}}
\\&
\lec
\TT^{-1}
\Vert v-v_0 \Vert_{L^2_t H^{s-1}_x ( (0,2\TT)\times \Omegaf)}
+
\TT^{-1/2}
\Vert v_0 \Vert_{H^s}
+
\Vert v \Vert_{K^{s+1}}
,
\end{split}
\label{EQ78}
\end{align}
since $s>3/2$.
An application of the fundamental theorem of calculus leads to
\begin{align}
\begin{split}
\Vert v-v_0\Vert_{L^2_t H^{s-1}_x ( (0,2\TT)\times \Omegaf)}
&\lec
\left\Vert \int_0^t v_t \,d\tau 
\right\Vert_{L^2_t H^{s-1}_x ( (0,2\TT)\times \Omegaf)}
\lec
\TT \Vert v\Vert_{K^{s+1}}
,
\llabel{EQ976}
\end{split}
\end{align}
where we also used the Cauchy-Schwarz inequality.
For the term $\mathcal{I}_4$, we proceed analogously to~\eqref{EQ153} to get
\begin{align}
\begin{split}
\mathcal{I}_4
&
\lec
\TT^{1/(s+1)}
\Vert v\Vert_{K^{s+1}}
.
\label{EQ971}
\end{split}
\end{align}
Inserting \eqref{EQ39} and \eqref{EQ78}--\eqref{EQ971} to \eqref{EQ970} and using the Young's inequality, we arrive at
\begin{align}
\begin{split}
&
\Vert \bar{w}\Vert_{H^{s/2+3/4}_t L^2_x (\Gammac)}
\lec
(\Vert v\Vert_{K^{s+1}} 
+
\TT^{-1/2} \Vert v_0\Vert_{H^s}
)^{(2s-1)/4}
( \TT^{1/(s+1)} \Vert v\Vert_{K^{s+1}}
)^{(5-2s)/4}
+
\TT^{1/2}
\Vert v\Vert_{K^{s+1}}
\\&\indeq\indeq
\lec
(\epsilon + \TT^{(5-2s)/4(s+1)} )
\Vert v\Vert_{K^{s+1}}
+
C_\epsilon 
\TT^{(-2s^2 -5s+11)/2(s+1)(2s-1)}
\Vert v_0\Vert_{H^s}
+
\TT^{1/2}
\Vert v\Vert_{K^{s+1}}
\\&\indeq\indeq
\lec
(\epsilon + \TT^{1/6} )
\Vert v\Vert_{K^{s+1}}
+
C_\epsilon 
\TT^{-1/7}
\Vert v_0\Vert_{H^s}
 ,
 \label{EQ201}
\end{split}
\end{align}
for any $\epsilon \in (0,1]$. 
Note that last inequality follows since $\epsilon_0 \in (0, 1/20)$.
From \eqref{EQ39}, \eqref{EQ353}, and \eqref{EQ201} we arrive at
\begin{align}
\left\Vert \frac{\partial \bar{w}}{\partial N}\right\Vert_{H^{s/2-1/4}_t L_x^2(\Gammac)}
\lec
\Vert w_0\Vert_{H^{s+1/2}}
+
\Vert w_1\Vert_{H^{s-1/2} }
+
(\epsilon + \TT^{1/6})
\Vert v\Vert_{K^{s+1}}
+
C_{\epsilon}
\TT^{-1/7}
 \Vert v_0\Vert_{H^s}
,
\label{EQ973}
\end{align}
for any $\epsilon\in (0,1]$.

Finally, inserting \eqref{EQ41} and \eqref{EQ973} to \eqref{EQ29}, we obtain
\begin{align}
\begin{split}
& 
\Vert \bar v \Vert_{K^{ s+1}}
+ 
\Vert \bar q \Vert_{H^{s/2-1/2, s}}
+ 
\Vert \nabla \bar q \Vert_{K^{s-1}}
+ 
\Vert \bar q  \Vert_{H^{s/2-1/4, s-1/2}_{\Gammac}}
\\&\indeq 
\leq
C (\epsilon 
+ 
\TT^{1/6}) 
\Vert v\Vert_{K^{s+1}}
+
C_\epsilon
\TT^{-1/7}
\Vert
v_{0} \Vert_{H^{s}} 
+ 
C
(\Vert w_{0} \Vert_{H^{s+1/2}} 
+ 
\Vert w_{1} \Vert_{H^{s-1/2}}
\\&\indeq\indeq
+  
\Vert f \Vert_{K^{ s-1}}
+ 
\Vert g \Vert_{K^{ s}}
+ 
\Vert \AA \Vert_{K^{ s-1}} 
+ 
\Vert \BB  \Vert_{K^{ s-1}} 
+ 
\Vert h \Vert_{H^{s/2-1/4, s-1/2}_{\Gammac}} )
,
\end{split}
\label{EQ30}
\end{align}
for any $\epsilon\in (0,1]$.
Let $T_0 = (1/4C)^6$, where $C$ is the constant in \eqref{EQ30}.
Then setting $\TT\in(0,T_0]$ and $\epsilon = 1/4C$, we obtain
\eqref{EQ500} for a sufficiently large constant $\CCC\geq 1$ depending on $C$.
Therefore, we have shown that the mapping $\Lambda\colon v\mapsto \bar{v}$
is well-defined from $\mathcal{Z}$ to $\mathcal{Z}$.

\colb
Let $v_1, v_2 \in \mathcal{Z}$ and $(\bar{w}_1, \bar{v}_1, \bar{q}_1)$ and $(\bar{w}_2, \bar{v}_2, \bar{q}_2)$ be the corresponding solutions of \eqref{EQ385}--\eqref{EQ488} and \eqref{EQ83}--\eqref{EQ190}, with the same initial data $(v_0, w_0, w_1)$ and the same nonhomogeneous terms $(f,g,\tilde{g}, b, h)$.
We proceed analogously as in \eqref{EQ29}--\eqref{EQ30} with
$(\bar{v}_1 - \bar{v}_2, \bar{q}_1 - \bar{q}_2)$ to obtain
  \begin{align}
  \begin{split}
  & 
  \Vert \bar v_2-\bar v_1 \Vert_{K^{ s+1}}
  + 
  \Vert \bar q_2-\bar q_1 \Vert_{H^{s/2-1/2, s}}
  + 
  \Vert \nabla (\bar q_2 - \bar q_1 )\Vert_{K^{s-1}}
  + 
  \Vert \bar q_2-\bar q_1  \Vert_{H^{s/2-1/4, s-1/2}_{\Gammac}}
  \\&\indeq 
  \leq
  \frac12
  \Vert v_2-v_1\Vert_{K^{s+1}}
  \end{split}
  \llabel{EQ26}
  \end{align}
and this, in particular, implies that
  \begin{align}
  \begin{split}
   &
  \Vert\Lambda( v_2- v_1) \Vert_{K^{ s+1}}
  \leq
  \frac12
  \Vert v_2-v_1\Vert_{K^{s+1}}.
  \end{split}
   \label{EQ52}
  \end{align}

On the other hand, using the hidden regularity Lemma~\ref{L03}, we have the interior estimate
\begin{align}
\begin{split}
&
\Vert \bar{w} \Vert_{C([0, 1], H^{s/2+5/4} (\Omegae))}
+
\Vert \bar{w}_t \Vert_{C([0, 1], H^{s/2+1/4}(\Omegae))}
\\&\indeq
\lec
\Vert w_0\Vert_{H^{s+1/2} (\Omegae)}
+
\Vert w_1\Vert_{H^{s-1/2} (\Omegae)}
+
\Vert \bar{w}\Vert_{H^{s/2+5/4, s/2+5/4} (\Gammac)}
,
\label{EQ380}
\end{split}
\end{align}
where we used $s>3/2$ in the last inequality.
For the last term on the right side of \eqref{EQ380}, we appeal to Lemma~\ref{L01} to get
  \begin{align}
  \begin{split}
  \Vert \bar{w}\Vert_{H^{s/2+5/4, s/2+5/4 } ( \Gammac)}
  &
  \lec
  \Vert \bar{w}_t \Vert_{H^{s/2+1/4}_t L^2_x (\Gammac) }
  + \Vert \bar{w} \Vert_{L^2_t L^2_x (\Gammac) }
  +
  \Vert \bar{w} \Vert_{L^2_t H^{s/2+5/4 }_x ( \Gammac)}
  \\&
  \lec_{\TT}
  \Vert v\Vert_{K^{s+1}}
  +
  \Vert v\Vert_{L^2_t H^{s+1}_x}
  \lec_{\TT}
  \Vert v\Vert_{K^{s+1}}
  .
  \end{split}
  \label{EQ42}
  \end{align}
Applying \eqref{EQ380}--\eqref{EQ42} for the difference $\bar{w}_1 - \bar{w}_2 = \tLambda v_2 - \tLambda v_1$, we get 
  \begin{align}
  \begin{split}
%
  \Vert \tLambda (v_2-v_1) \Vert_{C([0, 1], H^{s/2+5/4} (\Omegae))}
  +
  \Vert \partial_{t} (\tLambda (v_2-v_1)) \Vert_{C([0, 1], H^{s/2+1/4}(\Omegae))}
  &
  \lec_{\TT}
  \Vert v_2-v_1\Vert_{K^{s+1}}
  .
  \end{split}
   \label{EQ445}
  \end{align}
\colb
Now, to pass to the limit, we form the sequence
  \begin{align}
  \begin{split}
   v^{(n+1)} = \Lambda v^{(n)}
   \and
   w^{(n+1)} = \tLambda v^{(n)}
    \comma n\in\mathbb{N}_0
  \end{split}
   \label{EQ449}
  \end{align}
with
  \begin{align}
  \begin{split}
   v^{(0)} = v_0.
  \end{split}
   \label{EQ49}
  \end{align}
Then, passing to the limit, we obtain that there
exists a unique solution
\begin{align*}
\begin{split}
(v,q,w, w_{t} ) 
&
\in K^{s+1} ((0, 1)\times \Omegaf)
\times     H^{s/2-1/2, s}((0, 1)\times \Omegaf) \\&\indeq\indeq
\times
C([0, 1],H^{s/2+5/4}(\Omegae)) \times
C([0, 1], H^{s/2+1/4}(\Omegae) )   
,
\end{split}
\end{align*}
to the system \eqref{EQ218}--\eqref{EQ322}.
\end{proof}

\colb
\startnewsection{Solution to the fluid-structure system}{sec06}
In this section, we provide the local existence for the coupled Stokes-wave system \eqref{EQ01}--\eqref{EQ04} with the boundary conditions \eqref{EQ07}--\eqref{EQ08} and the initial data \eqref{EQ10}
by again suitably modifying it
to avoid issues with dependence of constants for small time.

Let $\psi_{\TT} (t)$ be a smooth cutoff function as in Section~\ref{sec05}, where $\TT\in (0,1/4]$ is a constant to be determined.
Instead of \eqref{EQ01}--\eqref{EQ08}, we consider the system
  \begin{align}
  &\partial_{t}v_{k} 
  -  a_{jl} \partial_{j}( a_{ml}
  \partial_{m} v_{k}) + a_{jk}\partial_{j}  q = 0
  \inon{in~$(0,1)\times\Omegaf $}
  \label{EQ01a}
  \\&
  a_{ji}\partial_{j}v_{i}=0
  \inon{in~$(0,1)\times\Omegaf    $}
  .
  \label{EQ02a}
  \end{align}
Here,
  \begin{align}
  \eta (t,x)
  =
  x
  +
  \int_{0}^{t} \psi_{\TT} (\tau) v (\tau, x) d\tau
  \inon{in~$[0,1]\times \Omegaf$}
  \label{EQ211}
  \end{align}
is a modified Lagrangian flow map and
$a=\cof(\nabla \eta)$ its cofactor matrix, i.e.,
  \begin{equation}
  a \nabla \eta 
  = 
  \nabla \eta a
  = 
  \det (\nabla \eta) 
  \mathbb{I}_3
  \inon{in~$[0,1]\times \Omegaf$}
  .
  \label{EQ100}
  \end{equation}
The displacement satisfies
\begin{equation}
w_{tt}
- \Delta w
= 0
\inon{in~$(0,1)\times\Omegae $}
,
\label{EQ04a}
\end{equation}
while the 
interaction boundary conditions on
$[0,1)\times\Gammac$ read as
\begin{align}
&
w(t,x) 
= 
w_0(x)
+
\int_0^t
\psi_{\TT}(\tau)
v(\tau,x) \,d\tau
\inon{on  $ [0,1)\times\Gammac    $},
\label{EQ320a}
\\&
\frac{\partial w_k}{\partial N} = a_{jl} a_{ml} \partial_{m} v_{k} N^{j} - a_{jk} q N^{j}
\inon{on  $[0,1)\times\Gammac    $}
    \comma k=1,2,3,
\label{EQ335}
\\&
v =0    \inon{on  $ [0,1)\times\Gammaf    $}.
\label{EQ322a}
\end{align}
In addition, we also have the side boundary conditions
  \begin{align}
  v(t,\cdot),
  q(t,\cdot),
  w(t,\cdot)
  ~~\text{periodic~in~the~$y_1$~and~$y_2$~directions}
  \label{EQ08a}
\end{align}
and the initial conditions
  \begin{align}
  \begin{split}
  & (v(0,\cdot), w(0,\cdot), w_{t}(0,\cdot))
  = (v_{0},w_0,w_1)
  \inon{in~ $ \Omegaf\times\Omegae\times\Omegae$}
  .
  \label{EQ10a}
  \end{split}
  \end{align}

In order to apply the fixed point argument, we
consider the inequality
\begin{align}
\begin{split}
&
\Vert v \Vert_{K^{ s+1} ((0,1) \times \Omegaf) }
\leq
   C
\bigl(
1 +   \Vert  v_{0} \Vert_{H^{s}}^{7}
+ 
\Vert w_{0} \Vert_{H^{s+1/2}} 
+ 
\Vert w_{1} \Vert_{H^{s-1/2}}
\bigr) 
,
\end{split}
   \label{EQ81}
\end{align}
where $C\geq1$ is a sufficiently large constant to be determined below.
Denote
\begin{align}
\begin{split}
\mathcal{Z}
&
= 
\Bigl\{v \in K^{s+1} ((0,1) \times \Omegaf):
v(0) = v_0 ~~\text{in}~~\Omegaf, 
v=0 ~~\text{on}~~ (0,1) \times\Gammaf, 
\\&\indeq\indeq\indeq
v~~\text{periodic in the $y_1$~and~$y_2$ directions},
~~\text{and \eqref{EQ81}~holds}\Bigr\}
.
\label{EQ389}
\end{split}
\end{align}
Let $v\in \mathcal{Z}$, we first obtain $\eta$ and $a$ using \eqref{EQ211} and \eqref{EQ100}. Then, we solve the wave equation \eqref{EQ04a} for $\tPi v=\bar{w}$ with the boundary and initial conditions \eqref{EQ387}--\eqref{EQ489}.
Next we define
\begin{align*}
\Pi \colon
v\in \mathcal{Z} \to \bar{v},
\end{align*}
where $(\bar{v}, \bar{q})$ is the solution to the linear Stokes equations \eqref{EQ83}--\eqref{EQ84} with the boundary and initial conditions \eqref{EQ86}--\eqref{EQ191}, where
\begin{align}
\begin{split}
&f_{k} 
=   
\partial_{j} 
( ( a_{jl} a_{ml} -
\delta_{jm}
) \partial_{m} \bar{v}_{k}  
)
-  
\partial_{j}(( a_{jk}  - \delta_{jk} ) \bar{q})
\\&
h_{k} 
=   
-  ( a_{jl} a_{ml} 
- \delta_{jm}
) \partial_{m}\bar{v}_{k} N^{j}
+
( a_{jk}  - \delta_{jk} ) \bar{q} N^{j}
\\
&g  
=  
-\partial_{j}  ( (a_{ji} - \delta_{ji} )\bar{v}_{i} ),
\label{EQ44}
\end{split}
\end{align}
for $k=1,2,3$.
Note that a fixed point of the system, if it exists, formally agrees with a solution to our original system \eqref{EQ01}--\eqref{EQ10} on $(0,\TT)$
due to \eqref{EQ007} and the Piola identity
\begin{equation}
\partial_{i} a_{ij} 
= 
0
,
\label{EQ79}
\end{equation}
for $j=1,2,3$.

\subsection{The Lagrangian map and cofactor matrix bounds}
Before we bound the terms in \eqref{EQ44}, we provide the bounds on the Lagrangian flow map and its cofactor matrix.

\cole
\begin{Lemma}
\label{L05}
Let $v \in K^{s+1} ((0,1) \times \Omegaf)$ and 
$\Vert v \Vert_{K^{s+1}} \leq M$ for some $M\geq 1$.
Let $\eta$ and $a$ be the Lagrangian flow map and its cofactor matrix, defined as in \eqref{EQ211}--\eqref{EQ100}.
Then for $\TT\in (0, 1/ M^6 ]$, we have
\begin{enumerate}[label=(\roman*)]
\item 
$
   \Vert \nabla \eta -\mathbb{I}_3 \Vert_{L^{\infty}_t H_x^{s}} 
   +
   \Vert a -\mathbb{I}_3 \Vert_{L^{\infty}_t H_x^{s}} 
   +
   \Vert aa^T -\mathbb{I}_3 \Vert_{L^{\infty}_t H_x^{s}} 
   \lec 
   \TT^{1/3}
$,
\item 
$
\Vert a - \mathbb{I}_3 \Vert_{H^{s-1}_t H_x^{s}}
+   
\Vert a a^T- \mathbb{I}_3 \Vert_{H^{s-1}_t H_x^{s}}
\lec 
\TT^{(5-3s)/6}
$,
\item  
$
\Vert a_t \bar{v} \Vert_{H^{(s-1)/2}_t L_x^{2}} 
\lec  
\TT^{(5-3s)/6} 
\Vert \bar{v}\Vert_{K^{s+1}}
$,
for any $\bar{v} \in K^{s+1}$.
\end{enumerate}
We emphasize that the implicit constants are independent of $M$ and $\TT$.
\end{Lemma}

\colb
\begin{proof}[Proof of Lemma~\ref{L05}]
(i)  An application of the fundamental theorem of calculus to \eqref{EQ211} leads to
\begin{align}
\begin{split}
\Vert \nabla \eta -\mathbb{I}_3 \Vert_{L^{\infty}_t H_x^{s}} 
\lec  
\int_{0}^{1} \Vert  \psi_{\TT} \nabla v \Vert_{H^{s}}
\, d\tau
\lec  
\int_{0}^{2\TT}  \Vert v \Vert_{H^{s+1}}
\, d\tau
\lec  
\TT^{1/2}
\Vert v \Vert_{L^2_t H_x^{s+1}}
\lec
\TT^{1/3} 
,
\end{split}
\label{EQ97}
\end{align}
where we used 
$\Vert v \Vert_{K^{s+1}} \leq M$
and
$\TT\in (0, 1/ M^6 ]$ in the last step.
For the cofactor matrix $a=\cof (\nabla\eta)$, we have the formula
\begin{equation}
a_{ij}
=
(\cof \nabla \eta)_{ij}
=
\frac12 \epsilon_{imn} \epsilon_{jkl} \partial_{m}\eta_k\partial_{n}\eta_l
\comma i,j=1,2,3,
\label{EQ107}
\end{equation}
where $\epsilon_{ijk}$ denotes the permutation (Levi-Civita) symbol.
From \eqref{EQ107}, we easily deduce
\begin{align}
\begin{split}
a_{ij} - \delta_{ij}
&=
\frac12 \epsilon_{imn} \epsilon_{jkl}
(\partial_{m}\eta_k-\delta_{mk})(\partial_{n}\eta_l-\delta_{nl})
+ \epsilon_{ikn}\epsilon_{jkl} (\partial_n\eta_l-\delta_{nl})
\comma i,j=1,2,3
,   
\end{split}
\label{EQ108}
\end{align}
by using the identities
$\epsilon_{ijk}\epsilon_{imn}=\delta_{jm}\delta_{kn}-\delta_{jn}\delta_{km}$
and
$\epsilon_{imn}\epsilon_{jmn}=2\delta_{ij}$. 
From \eqref{EQ97} and \eqref{EQ108} we obtain, with the help of the algebra property of $H^{s}$ for $s>3/2$,
\begin{align}
\begin{split}
\Vert a - \mathbb{I}_3 \Vert_{L^{\infty}_t H_x^{s}} 
\lec
\Vert \nabla\eta - \mathbb{I}_3\Vert_{L^{\infty}_t H_x^{s}}^2
+
\Vert \nabla\eta - \mathbb{I}_3\Vert_{L_t^{\infty} H_x^{s}}
\lec 
\TT^{1/3}
,
\end{split}
\llabel{EQ109}
\end{align}
where we used $\TT\lec1$.
Using the identity
$a a^{T}- \mathbb{I}_3=(a-\mathbb{I}_3) (a^{T}- \mathbb{I}_3)+(a-\mathbb{I}_3)+(a^{T}- \mathbb{I}_3)$, we deduce that
\begin{align}
\begin{split}
\Vert a a^{T}- \mathbb{I}_3 \Vert_{L^{\infty}_t H_x^{s}} 
&
\lec 
\Vert a- \mathbb{I}_3 \Vert_{L^{\infty}_t H_x^{s}}^2
+ 
\Vert a -\mathbb{I}_3 \Vert_{L^{\infty}_t H_x^{s}}
\lec
\TT^{1/3}
,
\end{split}
\label{EQ111}
\end{align}
thus completing the proof of (i).

(ii) In order to bound the first term, we use the Sobolev interpolation inequality, which gives
\begin{align}
\begin{split}
\Vert a - \mathbb{I}_3 \Vert_{H^{s-1}_t H_x^s}
&
\lec 
\Vert  a-\mathbb{I}_3 \Vert_{H^{ 1}_t H^s_x}^{s-1} 
\Vert  a-\mathbb{I}_3 \Vert_{L^{2}_t H^s_x}^{2-s}
.
\end{split}
\label{EQ98}
\end{align}
From \eqref{EQ107}, we write
\begin{align}
\begin{split}
\partial_{t}    
a_{ij} 
&=
\epsilon_{imn} \epsilon_{jkl}
\psi_{\TT}   \partial_{m} v_k \partial_{n}\eta_l   \comma i,j=1,2,3,
\llabel{EQ209}
\end{split}
\end{align}
which leads to
\begin{align}
\begin{split}
\Vert  a-\mathbb{I}_3 \Vert_{H^{ 1}_t H_x^s}
&
\lec
\Vert  \partial_{t}a \Vert_{L^{2}_t H_x^s}
+ 
\Vert  a-\mathbb{I}_3 \Vert_{L^{2}_t H_x^s}
\lec
\Vert \nabla v\Vert_{L^2_t H^s_x}
+
1
\lec
M
.
\end{split}
\label{EQ136}
\end{align}
From \eqref{EQ98} and \eqref{EQ136}, we deduce that
\begin{align}
  \begin{split}
\Vert a - \mathbb{I}_3 \Vert_{H^{s-1}_t H_x^s}
\lec
\TT^{(2-s)/3}
M^{s-1}
\lec
\TT^{(5-3s)/6}
,
  \end{split}
   \llabel{EQ24}
\end{align}
where we used (i) in the last inequality.
The second term in (ii) is estimated analogously to \eqref{EQ111} as 
\begin{align}
\begin{split}
\Vert a a^{T}- \mathbb{I}_3 \Vert_{H^{s-1}_t H_x^{s}} 
&
\lec 
\Vert a- \mathbb{I}_3 \Vert_{H^{s-1}_t H_x^{s}}^2
+ 
\Vert a -\mathbb{I}_3 \Vert_{H^{s-1}_t H_x^{s}}
\lec
\TT^{(5-3s)/6}
.
\end{split}
   \llabel{EQ25}
\end{align}
Therefore, the proof of (ii) is completed.

(iii) We claim that,
for any $\delta_1, \delta_2>0$ and $0<\alpha\leq 1/2+\delta_1$, we have
\begin{align}
\begin{split}
\Vert A B  \Vert_{H^{\alpha}_t L^2_x}   
&\lec
\Vert A   \Vert_{H^{1/2+\delta_1}_t L^{2}_x}
\Vert B   \Vert_{H^{\alpha}_t H^{3/2+\delta_2}_x}
\label{EQ403}
\end{split}
\end{align}
and
\begin{align}
\begin{split}
\Vert A B  \Vert_{H^{\alpha}_t L^2_x}   
&\lec
\Vert A   \Vert_{H^{1/2+\delta_1}_t H^{3/2+\delta_2}_x}
\Vert B   \Vert_{H^{\alpha}_t L^2_x}
\label{EQ404}
\end{split}
\end{align}
on the domain $[0,1]\times \Omegaf$.
Using extensions, we may assume that the domain is actually $\mathbb{R} \times \mathbb{R}^3$.
Then
\begin{align}
\begin{split}
\Vert A B  \Vert_{H^{\alpha}_t L^2_x}
&=    \Vert A B  \Vert_{L_{x}^{2}H^{\alpha}_t}
\lec  \bigl\Vert
\Vert A   \Vert_{H^{1/2+\delta_1}_t}
\Vert B   \Vert_{H^{\alpha}_t}
\bigr\Vert_{L^2_{x}}
\lec
\Vert A   \Vert_{L_{x}^{2} H^{1/2+\delta_1}_t}
\Vert B   \Vert_{L_{x}^\infty H^{\alpha}_t}
\\&
\lec
\Vert A   \Vert_{H_t^{1/2+\delta_1} L^{2}_x}
\Vert B   \Vert_{H^{3/2+\delta_2}_x H^{\alpha}_t}
\lec
\Vert A   \Vert_{H_t^{1/2+\delta_1} L^{2}_x}
\Vert B   \Vert_{ H^{\alpha}_t H^{3/2+\delta_2}_x}
.
\end{split}
\label{EQ379}
\end{align}
Therefore, the proof of the claim \eqref{EQ403} is completed. 
The proof of \eqref{EQ404} is similar as in \eqref{EQ379}, and thus we omit the details.
Using \eqref{EQ404}, we obtain
\begin{align}
 \Vert a_{t} \bar{v} \Vert_{H^{(s-1)/2}_t L_x^{2}} 
& 
\lec 
\Vert a_t \Vert_{H_t^{(s-1)/2}  L^2_x}
\Vert \bar{v}\Vert_{H_t^{1/2+(s-3/2)/4} H^{3/2+(s-3/2)/2}_x}
\lec
\Vert a_t \Vert_{H_t^{(s-1)/2}  L^2_x}
\Vert \bar{v} \Vert_{K^{s+1}},
\llabel{EQ305}
\end{align}
where we used Lemma~\ref{L02} in the last step.
To estimate the term $\Vert a_t \Vert_{H_t^{(s-1)/2} L^2_x}$, we appeal to \eqref{EQ403} and the Sobolev interpolation, obtaining
\begin{align}
\begin{split}
\Vert a_t \Vert_{H_t^{(s-1)/2} L^2_x}
&
\lec
\Vert \psi_{\TT} \nabla \eta \Vert_{H^{(s-1)/2}_t H^s_x} 
\Vert \nabla v \Vert_{H^{s-1}_t L^2_x}
\lec
\Vert \psi_{\TT} \nabla \eta \Vert_{H^{1}_t H^s_x}^{(s-1)/2} 
\Vert \psi_{\TT} \nabla \eta \Vert_{L^{2}_t H^s_x}^{(3-s)/2} 
\Vert v \Vert_{K^{s+1}}
\\&
\lec
(\TT^{-1/2} + M)^{(s-1)/2} \TT^{(3-s)/4}  M
\lec
\TT^{(5-3s)/6} 
.
\llabel{EQ301}
\end{split}
\end{align}
The proof of (iii) is thus complete.
\end{proof}

\subsection{Uniform boundedness of the sequence}
We denote the right side of \eqref{EQ81} by $M$, i.e.,
\begin{align}
	M:=    C
	\bigl(
	1 +   \Vert  v_{0} \Vert_{H^{s}}^{7}
	+ 
	\Vert w_{0} \Vert_{H^{s+1/2}} 
	+ 
	\Vert w_{1} \Vert_{H^{s-1/2}}
	\bigr) .
	\label{EQ681}
\end{align}
Suppose that $\Vert v\Vert_{K^{s+1}} \leq M$.
We shall prove that the mapping $\Pi$ is well-defined from $\mathcal{Z}$ to $\mathcal{Z}$, for some sufficiently large $C>0$ and $\TT=1/M^6$.
Note that the inequality \eqref{EQ90} holds
with $f$, $g$, $h$ as in \eqref{EQ44}, $\tilde g=0$, and
\begin{equation}
b_i
=  
\underbrace{\partial_{t} a_{ji} \bar{v}_{i} }_{\mathcal{B}_1}
+  
\underbrace{( a_{ji}  - \delta_{ji}) \partial_{t} \bar{v}_{i}}_{\mathcal{B}_2}
\comma
i=1,2,3.
\label{EQ46}
\end{equation}

%
We start with the term
$\Vert f \Vert_{K^{ s-1}}$.
For the space component, we use the algebra property of $H^{s}$ for $s>3/2$ to write
\begin{align}
\begin{split}
&
\Vert \partial_{j} 
\bigl( ( a_{jl}  a_{ml} 
- \delta_{jm}
) \partial_{m} \bar{v}_{k}
\bigr)
\Vert_{L^{2}_t H_x^{s-1}}
+
\Vert \partial_{j} (( a_{jk} - \delta_{jk} )  \bar{q} ) \Vert_{L^{2}_t H_x^{s-1}}
\\&\indeq
\lec 
\Vert a a^{T}  - \mathbb{I}_3 \Vert_{L^{\infty}_t H_x^{s}}
\Vert  \nabla \bar{v} \Vert_{L^{2}_t H_x^{s}}
+
\Vert a  - \mathbb{I}_3 \Vert_{L^{\infty}_t H_x^{s}}
\Vert \nabla \bar{q} \Vert_{L^{2}_t H_x^{s-1}}
\\&\indeq
\lec
\TT^{1/3}
\Vert \bar{v}\Vert_{K^{s+1}}
+
\TT^{1/3} 
\Vert \nabla \bar{q} \Vert_{K^{s-1}}
,
\end{split}
\label{EQ50}
\end{align}
where we used the Piola identity~\eqref{EQ79} in the first step and Lemma~\ref{L05} in the last step.
For the time component, we use Lemma~\ref{L02} to bound the velocity component as
\begin{align}
\begin{split}
&\Vert \partial_{j} ( 
( a_{jl}  a_{ml} 
- \delta_{jm}
) \partial_{m} \bar{v}_{k}  
) 
\Vert_{H_t^{(s-1)/2} L_x^2}
\lec
\Vert 
( a_{jl}  a_{ml} 
- \delta_{jm}
) \partial_{m} \bar{v}_{k} 
\Vert_{H^{(s-1)/2}_t H_x^{1}}
\\&\indeq
\lec
\Vert  ( a_{jl}  a_{ml} 
- \delta_{jm}
) \partial_{m} \bar{v}_{k}   \Vert_{L^{2}_t H_x^{s}}
+  
\Vert  ( a_{jl}  a_{ml} 
- \delta_{jm}
) 
\partial_{m} \bar{v}_{k}   \Vert_{H^{s/2}_t L_x^2}
\\&\indeq
\lec \Vert a a^{T}  - \mathbb{I}_3 \Vert_{L^{\infty}_t H_x^{s}}
\Vert  \nabla \bar{v} \Vert_{L^{2}_t H_x^{s}}
+ 
\Vert a a^{T}  - \mathbb{I}_3 \Vert_{H^{s-1}_t H_x^{s}}
\Vert  \nabla \bar{v} \Vert_{H^{s/2}_t L_x^{2}}
\\&\indeq 
\lec 
\TT^{(5-3s)/6}
\Vert \bar{v}\Vert_{K^{s+1}}
,
\end{split}
\label{EQ51}
\end{align}
where the third inequality follows from \eqref{EQ404}.
Similarly, the pressure part is estimated as
  \begin{align}
  \begin{split}
  &
   \Vert  \partial_{j}  (( a_{jk} - \delta_{jk} ) \bar{q})
   \Vert_{H^{(s-1)/2}_t L_x^{2}}
   \lec
   \Vert  ( a_{jk} - \delta_{jk} )\partial_{j} \bar{q}
   \Vert_{H^{(s-1)/2}_t L_x^{2}}
   \\&\indeq
    \lec
    \Vert a-\mathbb{I}_3 \Vert_{H^{s-1}_t H^s_x}
    \Vert \nabla \bar{q}\Vert_{H^{(s-1)/2}_t L^2_x}
        \lec
   \TT^{(5-3s)/6} 
   \Vert \nabla \bar{q}\Vert_{K^{s-1}}
    .
    \end{split}
    \label{EQ53}
  \end{align}
The divergence term  $\Vert g\Vert_{K^s}$ is estimated analogously to \eqref{EQ51} and \eqref{EQ53} as
\begin{align}
\begin{split}
\Vert \partial_{j}( ( a_{ji}  -     \delta_{ji})  \bar{v}_{i})  \Vert_{L^{2}_t H_x^{s}}
\lec
\Vert  ( a_{ji}  -     \delta_{ji})  \partial_{j} \bar{v}_{i}  \Vert_{L_t^{2} H_x^{s}}
\lec 
\Vert a   - \mathbb{I}_3 \Vert_{L_t^{\infty} H_x^{s}}
\Vert \nabla  \bar{v} \Vert_{L^{2}_t H_x^{s}}
\lec
\TT^{1/3} 
\Vert \bar{v}\Vert_{K^{s+1}}
\llabel{EQ54}
\end{split}
\end{align}
and
\begin{align}
\begin{split}
\Vert \partial_{j}( ( a_{ji}  -  \delta_{ji}) \bar{v}_{i} ) \Vert_{H^{s/2}_t L_x^{2}}
&
\lec
 \Vert ( a_{ji}  -  \delta_{ji}) \partial_{j} \bar{v}_{i}  \Vert_{H^{s/2}_t L_x^{2}}
\lec 
\Vert a-\mathbb{I}_3 \Vert_{H^{s-1}_t H^s_x}
\Vert \nabla \bar{v}\Vert_{H^{s/2}_t L^2_x}
\lec
\TT^{(5-3s)/6} 
\Vert \bar{v}\Vert_{K^{s+1}}
.
\llabel{EQ55}
\end{split}
\end{align}

Now, we turn to the terms $\mathcal{B}_1$ and $\mathcal{B}_2$ given in \eqref{EQ46}.
For the space part of the norm of $\mathcal{B}_1$, we appeal to H\"older's and Sobolev inequalities to obtain
\begin{align}
\begin{split}
\Vert \mathcal{B}_{1}
\Vert_{L^{2}_t H_x^{s-1}}
&
\lec 
\Vert \psi_{\TT} \nabla \eta \nabla v \bar{v} \Vert_{L^{2}_t H_x^{s-1}} 
\lec
\Vert \psi_{\TT} \nabla \eta \Vert_{L^{4}_t H_x^{s}} 
\Vert \nabla v \Vert_{L^{8}_t H_x^{s-1}} 
\Vert  \bar{v}\Vert_{L^{8}_t H_x^{s}} 
\\&
\lec
\TT^{1/4}
\Vert  \nabla \eta\Vert_{L^{\infty}_t H_x^{s}}
\Vert v\Vert_{K^{s+1}}
\Vert \bar{v}\Vert_{K^{s+1}}
\lec
\TT^{1/12} 
\Vert \bar{v}\Vert_{K^{s+1}}
,
\end{split}
\label{EQ59}
\end{align}
while for the time component, we have 
\begin{align}
\begin{split}
\Vert \mathcal{B}_{1} 
\Vert_{H^{(s-1)/2}_t L_x^{2}}
\lec
\TT^{(5-3s)/6}
\Vert \bar{v}\Vert_{K^{s+1}},
\end{split}
\label{EQ57}
\end{align}
by (iii) in Lemma~\ref{L05}.
For the term $\mathcal{B}_2$, we proceed analogously to \eqref{EQ51}--\eqref{EQ53} and write
\begin{align}
\begin{split}
&
\Vert \mathcal{B}_{2}
\Vert_{L^{2}_t H_x^{s-1}}
\lec \Vert a- \mathbb{I}_3 \Vert_{L^{\infty}_t H_x^{s}}
\Vert  \bar{v}_{t} \Vert_{L^{2}_t H_x^{s-1}} 
\lec  
\TT^{1/3} 
\Vert \bar{v}\Vert_{K^{s+1}}
\end{split}
\label{EQ56}
\end{align}
and
\begin{align}
\begin{split}
\Vert \mathcal{B}_{2} \Vert_{H^{(s-1)/2}_t L_x^{2}}
&
\lec
\Vert a-\mathbb{I}_3 \Vert_{H^{s-1}_t H^{s}_x}
\Vert \bar{v}_t\Vert_{H^{(s-1)/2}_t L^2_x}
\lec
\TT^{(5-3s)/6} 
\Vert \bar{v}\Vert_{K^{s+1}}.
\end{split}
\label{EQ58}
\end{align}

Finally, we estimate the term $\Vert h\Vert_{H^{s/2-1/4, s-1/2}_{\Gammac}}$.
For the space component, we use the classical trace inequality and proceed as in \eqref{EQ50}, obtaining
\begin{align}
\begin{split}
&
\Vert
( a_{jl} a_{ml} 
- \delta_{jm}
) \partial_{m} \bar{v}_k N^j
\Vert_{L^{2}_t H_x^{s-1/2} (\Gammac)}
+
\Vert 
( a_{jk}  - \delta_{jk} ) \bar{q} N^j
\Vert_{L^{2}_t H_x^{s-1/2} (\Gammac)}
\\&\indeq
\lec
\Vert
( a_{jl} a_{ml} 
- \delta_{jm}
) \partial_{m} \bar{v}_k
\Vert_{L^{2}_t H_x^{s}}
+
\Vert 
( a_{jk}  - \delta_{jk} ) \bar{q}
\Vert_{L^{2}_t H_x^{s}}
\\&\indeq
\lec
\TT^{1/3} 
\Vert \bar{v}\Vert_{K^{s+1}}
+
\TT^{1/3}
\Vert \bar{q}\Vert_{H^{s/2-1/2, s}}
.
\end{split}
\label{EQ47}
\end{align}
For the time component, we estimate the velocity part using Lemma~\ref{L01} as
\begin{align} 
\begin{split}
&
\Vert 
( a_{jl} a_{ml} 
- \delta_{jm}
) \partial_{m} \bar{v}_k
\Vert_{H^{s/2-1/4}_t L_x^{2}(\Gammac)}
\lec
\Vert (aa^T- \mathbb{I}_3) \nabla \bar{v}\Vert_{H^{s/2}_t L^2_x}
+
\Vert (aa^T- \mathbb{I}_3) \nabla \bar{v}\Vert_{L^2_t H^{s}_x }
\\&\indeq
\lec
\Vert aa^T- \mathbb{I}_3 \Vert_{H^{s-1}_t H^s_x} 
\Vert
 \nabla \bar{v}\Vert_{H^{s/2}_t L^2_x}
 +
 \Vert aa^T- \mathbb{I}_3 \Vert_{L^\infty_t H^{s}_x } 
 \Vert
 \nabla \bar{v}\Vert_{L^2_t H^{s}_x }
 \\&\indeq
 \lec
 \TT^{(5-3s)/6} \Vert \bar{v}\Vert_{K^{s+1}}
.
\end{split}
\label{EQ183}
\end{align}
For the pressure term, we appeal to the multiplicative Sobolev inequality to get
\begin{align}
\begin{split}
&
\Vert (a-\mathbb{I}_3) \bar{q}\Vert_{H^{s/2-1/4}_t L_x^{2}(\Gammac)}
\\&\indeq
\lec
\Vert a-\mathbb{I}_3 \Vert_{H^{s-1}_t H^{s-1/2}_x (\Gammac)} 
\Vert \bar{q}\Vert_{H^{s/2 - 1/4}_t L_x^{2} (\Gammac)}
\lec
\TT^{(5-3s)/6}
\Vert \bar{q}\Vert_{H_{\Gammac}^{s/2-1/4, s-1/2}}
,
\end{split}
\label{EQ180}
\end{align}
where we used the trace inequality.

Recall that $\TT=1/M^{6}$
and $\Vert v \Vert_{K^{s+1}} \leq M$.
Using Lemmas~\ref{L10} and~\ref{L05}, we infer from \eqref{EQ90} that
\begin{align}
\begin{split}
& \Vert \bar{v} \Vert_{K^{ s+1}}
+ 
\Vert \bar{q} \Vert_{H^{s/2-1/2, s}}
+ 
\Vert \nabla \bar{q}  \Vert_{K^{s-1}}
+ 
\Vert \bar{q}  \Vert_{H^{s/2-1/4, s-1/2}_{\Gammac}}
\\&\indeq
\leq
\CCC
(M^{6/7}
\Vert
v_{0} \Vert_{H^{s}} 
+ 
\Vert w_{0} \Vert_{H^{s+1/2}} 
+ 
\Vert w_{1} \Vert_{H^{s-1/2}})
\\&\indeq\indeq
+  
\frac{\CCC}{M^{(5-3s)}}
(\Vert \bar{v}\Vert_{K^{s+1}}
+
\Vert \bar{q}\Vert_{H_{\Gammac}^{s/2-1/4, s-1/2}}
+
\Vert \bar{q}\Vert_{H^{s/2-1/2, s}}
+
\Vert \nabla \bar{q} \Vert_{K^{s-1}})
,
\end{split}
\llabel{EQ64}
\end{align}
from where, using the restriction on $s$ and taking $C\geq 1$ in \eqref{EQ681} sufficiently large,
\begin{align}
\begin{split}
& \Vert \bar{v} \Vert_{K^{ s+1}}
+ 
\Vert \bar{q} \Vert_{H^{s/2-1/2, s}}
+ 
\Vert \nabla \bar{q}  \Vert_{K^{s-1}}
+ 
\Vert \bar{q}  \Vert_{H^{s/2-1/4, s-1/2}_{\Gammac}}
\\&\indeq
\leq
2\CCC
(M^{6/7}
\Vert
v_{0} \Vert_{H^{s}} 
+ 
\Vert w_{0} \Vert_{H^{s+1/2}} 
+ 
\Vert w_{1} \Vert_{H^{s-1/2}})
.
\end{split}
   \llabel{EQ63}
\end{align}
In order to obtain uniform boundedness, we need to assure that
  \begin{equation}
    2\CCC(M^{6/7}
    \Vert
    v_{0} \Vert_{H^{s}} 
    + 
    \Vert w_{0} \Vert_{H^{s+1/2}} 
    + 
    \Vert w_{1} \Vert_{H^{s-1/2}})
  \leq M
   ,
   \label{EQ67}
  \end{equation}
which follows from the choice
  \begin{equation}
   M= (6\CCC)^7
   \bigl(
        1 +   \Vert  v_{0} \Vert_{H^{s}}^{7}
        + 
        \Vert w_{0} \Vert_{H^{s+1/2}} 
        + 
        \Vert w_{1} \Vert_{H^{s-1/2}}
        \bigr)
   \label{EQ69}
  \end{equation}
in \eqref{EQ681}.
Thus, we have shown that the mapping $\Pi$ is well-defined from $\mathcal{Z}$ to $\mathcal{Z}$.
\colb

\subsection{Contracting property}
In this section, we shall prove
\begin{align}
\Vert \Pi v_1- \Pi v_2 \Vert_{K^{s+1}}
\leq
\frac{1}{2} \Vert v_1 - v_2 \Vert_{K^{s+1}}
\comma
v_1, v_2 \in \mathcal{Z},
\llabel{EQ298}
\end{align}
where  $\Vert v_1\Vert_{K^{s+1}} \leq M$ and $\Vert v_2 \Vert_{K^{s+1}} \leq M$ and $\TT = 1/M^6$.
The constant $M$ is chosen to be bigger than the right hand side of \eqref{EQ69}, i.e,
\begin{align}
	M = C(
	1 +   \Vert  v_{0} \Vert_{H^{s}}^{7}
	+ 
	\Vert w_{0} \Vert_{H^{s+1/2}} 
	+ 
	\Vert w_{1} \Vert_{H^{s-1/2}}
	)
	\label{EQ699}
\end{align}
where $C\geq (6\bar{C})^7$ is a sufficiently large constant to be determined below. 

Let $v_1, v_2 \in \mathcal{Z}$.
Denote by $(\eta_1, \eta_2)$ and $(a_1, a_2)$ the corresponding Lagrangian flow maps and cofactor matrices as in \eqref{EQ211}--\eqref{EQ100}, respectively.
First we solve for $(\bar{w}_1, \bar{w}_2)$ from \eqref{EQ385}--\eqref{EQ488} with the same initial data \eqref{EQ489}.
To obtain the next iterate $(\bar{v}_1, \bar{q}_1)$ and $(\bar{v}_2, \bar{q}_2)$, we solve the linear Stokes equations \eqref{EQ83}--\eqref{EQ84} with the boundary conditions \eqref{EQ86}--\eqref{EQ190} and the same initial data \eqref{EQ191}.
The equation for the differences
$(V,Q,W,E,A)=(\bar{v}_1-\bar{v}_2, \bar{q}_1-\bar{q}_2, \bar{w}_1- \bar{w}_2, \eta_1-\eta_2, a_1-a_2)$ reads
  \begin{align}
  &
  V_t
    -  \Delta V
    + \nabla  Q =  F
  \inon{in  $ (0,1)\times\Omegaf    $}
   \llabel{EQ140}
    \\&
    \div V = G
   \inon{in  $ (0,1)\times\Omegaf     $}
   \llabel{EQ141}
   \\&
   W_{tt} 
     - \Delta W=0    \inon{in  $ (0,1)\times\Omegae    $}
   \llabel{EQ142}
  \end{align}
subject to the boundary and initial conditions
  \begin{align}
   &
   W(t,x)
   =
   \int_0^t
   \psi_{\TT}(\tau) (v_1 (\tau, x) -v_2 (\tau, x)) \, d\tau
     \inon{on  $ (0,1)\times\Gammac    $}
   \llabel{EQ143}
   \\&
   \frac{\partial V}{\partial N} - Q N
   = \frac{\partial W}{\partial N} + H
   \inon{on  $ (0,1)\times\Gammac    $}
   \llabel{EQ144}
   \\&
   V=0    \inon{on  $ (0,1)\times\Gammaf    $}
   \llabel{EQ145}
   \\&
   W, V, Q, E, A~~
   \text{periodic in the $y_1$~and~$y_2$ directions}
   \llabel{EQ146}
   \\&
   (W, W_t, V) (0) = (0, 0, 0),
   \llabel{EQ889}
  \end{align}
where
  \begin{equation}
   (F,G,H)=(f_1-f_2, g_1-g_2, h_1-h_2)
   ;
   \llabel{EQ168}
  \end{equation}
here 
$(f_1, g_1, h_1)$ and
$(f_2, g_2, h_2)$ are the forcing terms in \eqref{EQ44} corresponding to $(a_1, \eta_1, \bar{v}_1, \bar{q}_1)$ and $(a_2, \eta_2, \bar{v}_2, \bar{q}_2)$, respectively.

We first derive estimates on the difference of the Lagrangian flow map and cofactor matrix.
\cole
\begin{Lemma}
\label{L07}
Let $v_1, v_2\in K^{s+1} ((0,1)\times \Omegaf)$.
Suppose that $\Vert v_1\Vert_{K^{s+1}}\leq M$ and $\Vert v_2\Vert_{K^{s+1}}\leq M$ for some $M\geq 1$. Then for $\TT\in (0,1/M^6]$, we have
\begin{enumerate}[label=(\roman*)]
\item $\Vert \nabla E\Vert_{L^{\infty}_t H_x^{s}} +\Vert A \Vert_{L^{\infty}_t H_x^{s}} +\Vert a_1 a_1^{T} - a_2 a_2^{T} \Vert_{L^{\infty}_t H_x^{s}} 
\leq \TT^{1/2}  \Vert v_1 -v_2 \Vert_{K^{s+1}}$,
\item $\Vert A  \Vert_{H^{s-1}_t H^{s}_x} +\Vert a_1 a_1^{T}
- a_2 a_2^{T}  \Vert_{H^{s-1}_t H^{s}_x} 
      \leq 
      \TT^{(2-s)/2} 
      \Vert v_1 - v_2\Vert_{K^{s+1}}$.
\end{enumerate}
We emphasize that the implicit constants are independent of $M$ and $\TT$.
\end{Lemma}

\colb
\begin{proof}[Proof of Lemma~\ref{L07}]
(i) Similarly to \eqref{EQ97}, we appeal to the Cauchy-Schwarz inequality to get
\begin{align}
\begin{split}
\Vert \nabla E\Vert_{L^{\infty}_t H_x^{s}}
\lec
\int_0^{2\TT} \Vert \nabla (v_1 -v_2) \Vert_{H^{s}}
\, d\tau
\lec
\TT^{1/2} \Vert v_1 -v_2\Vert_{K^{s+1}}
.
\label{EQ888}
\end{split}
\end{align}
From \eqref{EQ107} and \eqref{EQ888} it follows that
\begin{align}
\begin{split}
\Vert A\Vert_{L^{\infty}_t H_x^{s}} 
&
\lec 
\Vert \nabla E \Vert_{L^\infty_t H_x^{s}} 
( \Vert \nabla \eta_1 \Vert_{L^{\infty}_t H_x^{s}}  
+ \Vert \nabla \eta_2 \Vert_{L^\infty_t H_x^{s}}
) 
\lec 
\TT^{1/2} \Vert v_1 -v_2 \Vert_{K^{s+1}}
\end{split}
\label{EQ66}
\end{align} 
and
\begin{align}
\begin{split}
\Vert a_1 a_1^{T} 
- 
a_2 a_2^{T} \Vert_{L^{\infty}_t H_x^{s}} 
&
\lec  
\Vert A \Vert_{L^{\infty}_t H_x^{s}}  
( 
\Vert a_1 \Vert_{L^{\infty}_t H_x^{s}}
+ 
\Vert a_2 \Vert_{L^{\infty}_t H_x^{s}}
) 
\lec 
\TT^{1/2} 
\Vert v_1 -v_2 \Vert_{K^{s+1}}
,
\end{split}
\label{EQ68}
\end{align}
where we also used Lemma~\ref{L05}.

(ii) We proceed analogously to \eqref{EQ66} and \eqref{EQ68}, obtaining
\begin{align*}
\begin{split}
\Vert A \Vert_{H^{s-1}_t H^{s}_x } 
&
\lec 
\Vert A \Vert_{L^{2}_t H_x^{s}}^{2-s} 
\Vert A \Vert_{H^{1}_t H_x^{s}}^{s-1}
\\&
\lec
\TT^{(2-s)/2}
\Vert v_1 -v_2\Vert_{K^{s+1}}^{2-s} 
((\TT^{1/2} M+1) \Vert v_1 -v_2\Vert_{K^{s+1}})^{s-1}
+
\TT^{1/2} 
\Vert v_1 -v_2\Vert_{K^{s+1}}
\\&
\lec
\TT^{(2-s)/2} 
\Vert v_1 - v_2\Vert_{K^{s+1}}
\end{split}
\end{align*}
and
\begin{align*}
\begin{split}
&
\Vert a_1 a_1^T-a_2 a_2^T \Vert_{H^{s-1}_t H^{s}_x } 
\lec  
\Vert A \Vert_{H^{s-1}_t H_x^{s}}  
( 
\Vert a_1 \Vert_{H^{s-1}_t H_x^{s}}
+ 
\Vert a_2 \Vert_{H^{s-1}_t H_x^{s}}
) 
\lec
\TT^{(2-s)/2} 
\Vert v_1 - v_2\Vert_{K^{s+1}}
,
\end{split}
\end{align*} 
where we also used Lemma~\ref{L05}.
\end{proof}

\begin{proof}[Proof of Theorem~\ref{T01}]
First, we proceed as in \eqref{EQ29}--\eqref{EQ30} to get
\begin{align}
\begin{split}
&
\Vert V \Vert_{K^{ s+1}}
+ 
\Vert Q  \Vert_{H^{s/2-1/2, s}}
+ 
\Vert \nabla Q  \Vert_{K^{s-1}}
+ 
\Vert Q  \Vert_{H^{s/2-1/4, s-1/2}_{\Gammac}}
\\&\indeq
\lec 
\TT^{1/6}
\Vert v_1 -v_2\Vert_{K^{s+1}}
+
\Vert F \Vert_{K^{ s-1}}
+ 
\Vert G      \Vert_{K^{ s}}
+ 
\Vert H \Vert_{H^{s/2-1/4, s-1/2}_{\Gammac}}
+ 
\Vert B  \Vert_{K^{ s-1}} 
,
\end{split}
\llabel{EQ179}
\end{align}
where
\begin{align}
B
=
(\partial_{t} a_{1ji} \bar{v}_{1i} 
+  
( a_{1ji}  - \delta_{ji}) \partial_{t} \bar{v}_{1i})
-
(\partial_{t} a_{2ji} \bar{v}_{2i} 
+  
( a_{2ji}  - \delta_{ji}) \partial_{t} \bar{v}_{2i})
.
\label{EQ800}
\end{align}
For the space component of the norm of 
\begin{align}
\begin{split}
F_{k}
&
=  
\partial_{j}(
(a_{1jl} a_{1ml} - a_{2jl} a_{2ml}) \partial_m \bar{v}_{1k}
)
+
\partial_j (
(a_{2jl} a_{2ml} - \delta_{jm}) \partial_m V_k
)
\\&\indeq
-
(a_{1jk}- a_{2jk})
\partial_{j} \bar{q}_1
-
(a_{2jk}- \delta_{jk})
\partial_{j} Q
,
\end{split}
\llabel{EQ60}
\end{align}
where $k=1,2,3$, we have
\begin{align}
\begin{split}
\Vert F_k \Vert_{L^{2}_t H_x^{s-1}} 
&
\lec 
\Vert a_1 a_1^{T}  -  a_2 a_2^{T} \Vert_{L^{\infty}_t H_x^{s}}
\Vert  \bar{v}_1 \Vert_{L^{2}_t H_x^{s+1}} 
+    
\Vert a_2 a_2^{T} - \mathbb{I}_3 \Vert_{L^{\infty}_t H_x^{s}}   
\Vert  V \Vert_{L^{2}_t H_x^{s+1}} 
\\&\indeq
+    
\Vert A\Vert_{L^{\infty}_t H_x^{s}} 
\Vert  \bar{q}_1 \Vert_{L^{2}_t H_x^{s}} 
+    
\Vert a_2 - \mathbb{I}_3 \Vert_{L^{\infty}_t H_x^{s}}  
 \Vert Q \Vert_{L^{2}_t H_x^{s}} 
\\& 
\lec  
\TT^{1/3} 
(
\Vert v_1 - v_2\Vert_{K^{s+1}}
+
\Vert V\Vert_{K^{s+1}} 
+
\Vert Q \Vert_{H^{s/2-1/2, s}}
 )
,
\end{split}
\label{EQ70}
\end{align}
where we used Lemmas~\ref{L05}--\ref{L07} in the last step.
Next, we bound the time component of the norm 
of
$F_k$
similarly to \eqref{EQ51} and \eqref{EQ53} as
\begin{align}
\begin{split}
&
\Vert F_k \Vert_{H_t^{(s-1)/2} L_x^{2}} 
\lec
\TT^{(5-3s)/6} 
(
\Vert v_1 - v_2\Vert_{K^{s+1}}
+
\Vert V\Vert_{K^{s+1}} 
+
\Vert Q \Vert_{H^{s/2-1/2, s}}
)
,
\end{split}
\label{EQ71}
\end{align}
again by using Lemmas \ref{L05}--\ref{L07}.
The term $G$ is bounded similarly to \eqref{EQ70} and \eqref{EQ71} as
\begin{align}
\begin{split}
\Vert G \Vert_{L^{2}_t H_x^{s}}
&
\lec 
\Vert A \Vert_{L^{\infty}_t H_x^{s}} 
\Vert  \bar{v}_1 \Vert_{L^{2}_t H_x^{s+1}} 
+    
\Vert a_2 - \mathbb{I}_3 \Vert_{L^{\infty}_t H_x^{s}}  
\Vert V \Vert_{L^{2}_t H_x^{s+1}} 
\\&
\lec  
\TT^{1/3} 
(\Vert v_1 - v_2\Vert_{K^{s+1}}
+
 \Vert V \Vert_{K^{s+1}})
,
\end{split}
\llabel{EQ72}
\end{align}
and
\begin{align}
\begin{split}
\Vert G\Vert_{H^{s/2}_t L_x^{2}} 
&\lec 
\Vert A \Vert_{H^{s-1}_t H^{s}_x}
\Vert \nabla \bar{v}_1 \Vert_{H^{s/2}_t L_x^{2}}
+
 \Vert a_1 - \mathbb{I}_3 \Vert_{H^{ s-1}_t H^{s}_x}
\Vert \nabla V\Vert_{H^{s/2}_t L_x^{2}}
\\&
\lec 
\TT^{(5-3s)/6} 
(\Vert v_1 - v_2\Vert_{K^{s+1}}
+
\Vert V\Vert_{K^{s+1}})
.
\end{split}
\llabel{EQ73}
\end{align}
Next, we estimate the difference $H$. 
For the space component, we use,
in analogy with \eqref{EQ47},
\begin{align}
\begin{split}
\Vert H \Vert_{L^{2}_t H_x^{s-1/2}(\Gammac)}
&
\lec
\Vert a_1 a_1^{T}  - a_2 a_2^{T} \Vert_{L^{\infty}_t H_x^{s}}
\Vert  \nabla \bar{v}_1 \Vert_{L^2_t H^s_x}
+
\Vert a_2  a_2^{T}  - \mathbb{I}_3 \Vert_{L^{\infty}_t H_x^{s}}
\Vert  \nabla V\Vert_{L^2_t H^s_x}
\\&\indeq\indeq
+
\Vert A \Vert_{L^{\infty}_t H_x^{s}}
\Vert \bar{q}_1 \Vert_{L^2_t H_x^{s}} 
+
\Vert a_2 - \mathbb{I}_3 \Vert_{L^{\infty}_t H_x^{s}}
\Vert Q \Vert_{L^2_t H_x^{s}} 
\\&
\lec 
\TT^{1/3} 
(\Vert v_1 - v_2\Vert_{K^{s+1}}
+
\Vert V\Vert_{K^{s+1}} 
+
\Vert Q \Vert_{H^{s/2-1/2, s}}
)
,
\end{split}
\llabel{EQ74}
\end{align} 
while for the time component we have, similarly to \eqref{EQ183} and \eqref{EQ180},
\begin{align}
\begin{split}
\Vert H \Vert_{H^{s/2-1/4}_t L_x^{2}(\Gammac)} 
\lec
\TT^{(5-3s)/6}
(\Vert v_1 - v_2\Vert_{K^{s+1}}
+
\Vert V\Vert_{K^{s+1}} 
+
\Vert Q \Vert_{H^{s/2-1/2, s}}
+
\Vert Q \Vert_{H_{\Gammac}^{s/2-1/4, s-1/2}}
).
\end{split}
\llabel{EQ75}
\end{align} 
Finally, we estimate the term $B$ given in \eqref{EQ800}.
The space component of the norm is estimated,
in analogy with \eqref{EQ59} and \eqref{EQ56}, as
\begin{align}
\begin{split}
\Vert B \Vert_{L^{2}_t H_x^{s-1}}
&\lec 
\TT^{1/12}
(\Vert v_1 -v_2\Vert_{K^{s+1}}
+
 \Vert V\Vert_{K^{s+1}})
,
\end{split}
\llabel{EQ62}
\end{align}
while for the time component, analogously to 
\eqref{EQ57} and \eqref{EQ58},
\begin{align}
\begin{split}
\Vert B\Vert_{H^{(s-1)/2}_t L_x^{2}}   
&
\lec
\TT^{(5-3s)/6}
(\Vert V \Vert_{K^{s+1}}
+
\Vert v_1 - v_2\Vert_{K^{s+1}})
.
\end{split}
\llabel{EQ187}
\end{align}
Summarizing the above estimates, we arrive at
\begin{align}
\begin{split}
&
\Vert V \Vert_{K^{ s+1}}
+ 
\Vert Q  \Vert_{H^{s/2-1/2, s}}
+ 
\Vert \nabla Q  \Vert_{K^{s-1}}
+ 
\Vert Q \Vert_{H^{s/2-1/4, s-1/2}_{\Gammac}}
\\&\indeq
\lec
\TT^{1/6}
\Vert v_1 -v_2\Vert_{K^{s+1}}
+
\TT^{(5-3s)/6}
(\Vert v_1 - v_2\Vert_{K^{s+1}}
+
\Vert V\Vert_{K^{s+1}} 
+
\Vert Q \Vert_{H^{s/2-1/2, s}}
+
\Vert Q \Vert_{H_{\Gammac}^{s/2-1/4, s-1/2}}
).
\end{split}
\llabel{EQ77}
\end{align}
Taking the constant $C\geq (6\bar{C})^7$ in \eqref{EQ699} sufficiently large and thus $\TT = 1/M^6$ sufficiently small, we get
\begin{align*}
\begin{split}
\Vert \Pi(v_1 - v_2)\Vert_{K^{s+1}}
\leq
\frac{1}{2} 
\Vert v_1 - v_2\Vert_{K^{s+1}}
\comma
v_1, v_2 \in \mathcal{Z}.
\end{split}
\end{align*}
Therefore, we have established that the mapping $\Pi$ is contracting.
Then
 we proceed analogously to~\eqref{EQ380}--\eqref{EQ445} to obtain the interior regularity estimate 
\begin{align*}
	\Vert \tilde{\Pi} v \Vert_{C([0, 1], H^{s/2+5/4} (\Omegae))}
	+
	\Vert \partial_{t}(\tilde{\Pi} v) \Vert_{C([0, 1], H^{s/2+1/4}(\Omegae))}
	\les_{\TT}
	\Vert w_0\Vert_{H^{s+1/2} (\Omegae)}
	+
	\Vert w_1\Vert_{H^{s-1/2} (\Omegae)}
	+
	\Vert v\Vert_{K^{s+1}}.
\end{align*}
Applying the similar estimate for the difference $\bar{w}_1 - \bar{w}_2 = \tilde{\Pi} v_2 - \tilde{\Pi} v_1$, we arrive at
\begin{align}
	\begin{split}
		\Vert \tilde{\Pi} (v_2-v_1) \Vert_{C([0, 1], H^{s/2+5/4} (\Omegae))}
		+
		\Vert \partial_{t} (\tilde{\Pi} (v_2-v_1)) \Vert_{C([0, 1], H^{s/2+1/4}(\Omegae))}
		&
		\lec_{\TT}
		\Vert v_2-v_1\Vert_{K^{s+1}}
		.
	\end{split}
	\llabel{EQ645}
\end{align}
Now, to pass to the limit, we form the sequence
\begin{align}
	\begin{split}
		v^{(n+1)} = \Pi v^{(n)}
		\and
		w^{(n+1)} = \tilde{\Pi} v^{(n)}
		\comma n\in\mathbb{N}_0
	\end{split}
	\llabel{EQ649}
\end{align}
with
\begin{align}
	\begin{split}
		v^{(0)} = v_0.
	\end{split}
	\llabel{EQ749}
\end{align}
Then we pass to the limit to obtain that there exists a unique solution
\begin{align*}
	\begin{split}
		(v,q,w, w_t) 
		&
		\in 
		K^{s+1} ((0, 1)\times \Omegaf) \times  H^{s/2-1/2, s}  ((0, 1)\times \Omegaf)
		\\&\indeq\indeq
		\times C([0, 1], H^{s/2+ 5/4 }(\Omegae))
		\times
		C([0, 1], H^{s/2 +1/4 }(\Omegae))
	\end{split}
\end{align*}
to the system \eqref{EQ01a}--\eqref{EQ10a}.

Now we fix the constants $M$ as in \eqref{EQ699} and $\TT= 1/M^6$.
One may easily check that the limiting cofactor matrix $a$ is the inverse of $\nabla \eta$ on the time interval $[0,\TT ]$, where $\eta$ is the limiting Lagrangian map.
Consequently, there exists a unique solution
\begin{align*}
\begin{split}
(v,q,w, w_t) 
&
\in 
K^{s+1} ((0,\TT)\times \Omegaf) \times  H^{s/2-1/2, s}  ((0,\TT)\times \Omegaf)
\\&\indeq\indeq
 \times C([0,\TT], H^{s/2+ 5/4 }(\Omegae))
\times
C([0,\TT], H^{s/2 +1/4 }(\Omegae))
\end{split}
\end{align*}
to the system~\eqref{EQ01}--\eqref{EQ10},
which then gives a solution to the system~\eqref{EQ01a}--\eqref{EQ10a}.
\end{proof}

\colb
\section*{Acknowledgments}
IK was supported in part by the
NSF grant DMS-2205493.
Part of the work was completed while the authors were members of the MSRI program ``Mathematical problems in fluid dynamics'' during the Spring~2021 semester (NSF DMS-1928930).

\ifnum\sketches=1
\newpage
\begin{center}
  \bf   Notes \rm
\end{center}
\large
  \begin{equation}
   \Vert v  \Vert_{H^{\theta} H^{1}}    
   \lec
    \Vert  v  \Vert_{H^{\theta+1/2} L^{2}} 
    +   \Vert v \Vert_{ L^{2}H^{2\theta+1}} 
    ,
   \label{EQ121restate}
  \end{equation}
  \begin{align}
   \begin{split}
   \Vert \partial u/\partial N\Vert_{H^{s/2-1/4}L^2({\Gammac})}
   \lec
   \Vert \nabla u \Vert_{H^{s/2}L^{2}}
      + \Vert \nabla u \Vert_{L^{2} H^{s}}
   \end{split}
   \label{EQ153restate}
  \end{align}
  \begin{equation}
   \Vert \partial u/\partial N\Vert_{H^{s/2-1/4, s-1/2}_{\Gammac}}   
   \lec
   \Vert u\Vert_{K^{s+1}}
   .
   \label{EQ175restate}
  \end{equation}

\normalsize
\begin{itemize}
\item check all domains                            OK
\item check all equation references                OK
\item use checkref to remove unnecessary labels    OK
\item use the-the                                  OK
\item check parentheses                            OK
\item check introduction grammar carefully         OK
\item erase all comments                           OK
\item remove colors                                OK
\item remove unnecessary displayed formulas        OK
\item re-read the whole paper
check overfull boxes
\item check for commas and periods in all the displayed formulas
\item remove all $\backslash$newpage's             OK
\item comment out sketches                         OK
\item acknowledge the support of MSRI              OK
\item spell-check                                  OK
\item check for question marks in the pdf          OK
\end{itemize}

\fi
\end{document}